\documentclass[10pt]{amsart}

\usepackage{geometry,graphicx,amssymb,amsmath,amsbsy,eucal,amsfonts,mathrsfs,amscd,bm,esint,yhmath}
\usepackage[all]{xy}
\usepackage{tikz}
\usetikzlibrary{arrows,shapes}
\usetikzlibrary{arrows, decorations.markings,fit}
\usetikzlibrary{calc,3d}
\usepackage{tikz-3dplot}
\usetikzlibrary{positioning, quotes}

\tikzset{
    >=stealth',
    punkt/.style={
           rectangle,
           rounded corners,
           draw=black, very thick,
           text width=6.5em,
           minimum height=2em,
           text centered},
    pil/.style={
           ->,
           thick,
           shorten <=2pt,
           shorten >=2pt,}
}

\usepackage{listings}

\numberwithin{equation}{section}

\allowdisplaybreaks[3]

\newtheorem{theorem}{Theorem}[section]
\newtheorem{lemma}[theorem]{Lemma}

\newtheorem{proposition}[theorem]{Proposition}
\theoremstyle{definition}
\newtheorem{definition}[theorem]{Definition}
\theoremstyle{remark}
\newtheorem{remark}[theorem]{Remark}

\newcommand{\p}{{\partial}}

\newcommand{\nab}{\nabla}

\newcommand{\Div}{{\rm div}\,}

\newcommand{\pol}{\EuScript{P}}
\newcommand{\bpol}{\boldsymbol{\pol}}
\newcommand{\bld}[1]{\boldsymbol{#1}}

\newcommand{\bI}{\bld{I}}

\newcommand{\bv}{\bld{v}}
\newcommand{\bw}{\bld{w}}

\newcommand{\bp}{\bld{p}}
\newcommand{\bn}{\bld{n}}
\newcommand{\bu}{\bld{u}}

\newcommand{\bV}{\bld{V}}

\newcommand{\bPi}{\bld{\Pi}}

\newcommand{\bH}{\bld{H}}

\newcommand{\bx}{\bld{x}}
\newcommand{\bC}{\bld{C}}
\newcommand{\bL}{\bld{L}}

\newcommand{\bX}{{\bm X}}

\newcommand{\balpha}{{\bm \alpha}}

\newcommand{\calV}{\mathcal{V}}

\newcommand{\bbR}{\mathbb{R}}

\newcommand{\calS}{\mathcal{S}}
\newcommand{\calT}{\mathcal{T}}

\newcommand{\bnu}{{\bm \nu}}
\newcommand{\calP}{\mathcal{P}}
\newcommand{\calM}{\mathcal{M}}

\newcommand{\bell}{{\bm \ell}}

\newcommand{\pt}{\wideparen}

\makeatletter
\def\widebreve{\mathpalette\wide@breve}
\def\wide@breve#1#2{\sbox\z@{$#1#2$}%
     \mathop{\vbox{\m@th\ialign{##\crcr
\kern0.08em\brevefill#1{0.8\wd\z@}\crcr\noalign{\nointerlineskip}%
                    $\hss#1#2\hss$\crcr}}}\limits}
\def\brevefill#1#2{$\m@th\sbox\tw@{$#1($}%
  \hss\resizebox{#2}{\wd\tw@}{\rotatebox[origin=c]{90}{\upshape(}}\hss$}
\makeatletter

\newcommand{\ipt}{\breve}
\newcommand{\iptW}{\widebreve}

\title[FEM for surface Stokes]{A tangential and penalty-free finite element 
method for the surface Stokes problem}\thanks{The first author was partially supported by NSF grant DMS-2012326.  The second author 
was partially supported by NSF grants DMS-2011733 and DMS-2309425}

\author[A. Demlow and M. Neilan]{
Alan Demlow\address{Department of Mathematics, Texas A\&M University, College Station, TX, 77843}
\email{demlow@math.tamu.edu}
\and{Michael Neilan}
\address{Department of Mathematics, University of Pittsburgh, Pittsburgh, PA 15260}
\email{neilan@pitt.edu}}


\begin{document}

\maketitle
\begin{abstract}  
Surface Stokes and Navier-Stokes equations are used to model fluid flow on surfaces.  They have attracted 
significant recent attention in the numerical analysis literature because approximation of their solutions poses significant challenges 
not encountered in the Euclidean context.  One challenge comes from the need to simultaneously enforce tangentiality and 
$H^1$ conformity (continuity) of discrete vector fields used to approximate solutions in the velocity-pressure formulation.  
Existing methods in the literature all enforce one of these two constraints weakly either by penalization or by use of Lagrange multipliers.   
Missing so far is a robust and systematic construction of surface Stokes finite element spaces which employ 
nodal degrees of freedom, including MINI, Taylor-Hood, Scott-Vogelius, and other composite elements which can lead to 
divergence-conforming or pressure-robust discretizations.  In this paper we construct surface MINI spaces whose velocity fields are {tangential}.  
They are not $H^1$-conforming, but do lie in $H({\rm div})$ and do not require penalization 
to achieve optimal convergence rates.   We prove stability and optimal{-order} energy-norm convergence of the method and demonstrate optimal{-order} 
{convergence} of the velocity field in $L_2$ via numerical experiments.   The core advance in the paper is the 
construction of nodal degrees of freedom for the velocity field.  This technique also may be used to construct surface counterparts to 
many other standard Euclidean Stokes spaces, and we accordingly present numerical experiments indicating optimal{-order} convergence 
of nonconforming tangential surface Taylor-Hood $\mathbb{P}^2-\mathbb{P}^1$ elements.  
\end{abstract}

\pagestyle{myheadings}
\thispagestyle{plain}

\section{Introduction}

In this paper, we consider the surface Stokes problem:
\begin{subequations}
\label{eqn:Stokes}
\begin{alignat}{2}
\label{eqn:Stokes1}
-\bPi {\rm div}_\gamma {\rm Def}_\gamma \bu + \nab_\gamma p +\bu & = {\bm f}\qquad \text{on }\gamma,\\
\label{eqn:Stokes2}
{\rm div}_\gamma \bu & = 0\qquad \text{on }\gamma.
\end{alignat}
\end{subequations}
Here, $\gamma\subset \bbR^3$
 is a smooth and connected two-dimensional surface with outward unit
 normal $\bnu$,
 $\bPi = {\bf I}-\bnu\otimes \bnu$ is the orthogonal projection
 onto the tangent space of $\gamma$,
 and $\nab_\gamma$ and ${\rm div}_\gamma$ are the surface
 gradient and surface divergence operators, respectively.
 Furthermore,
 ${\rm Def}_\gamma$ is the tangential deformation operator,
 and the forcing function ${\bm f}$ is assumed to be tangential
 to the surface to ensure well-posedness.  Further assumptions
 and notation are given in Section \ref{sec-notation};  
 {cf.~\cite{JankuhnEtal18} for derivation of {the} surface Stokes {problem} and related models and further discussion of their properties.}
 The system of equations \eqref{eqn:Stokes} is subject to the
 tangential velocity constraint $\bu\cdot \bnu = 0$.
To address degeneracies related to Killing fields, i.e., non-trivial tangential vector fields in 
the kernel of ${\rm Def}_\gamma$, we include a zeroth-order mass term in 
the momentum equations \eqref{eqn:Stokes1} (cf.~Remark \ref{rem-Killing}).

We consider surface finite element methods (SFEMs),
a natural methodology mimicking the  variational formulation and built upon classical
Galerkin principles.  In this approach
the domain $\gamma$ is approximated by a polyhedral (or higher-order) surface $\Gamma_h$
whose faces constitute the finite element mesh. 
Similar to the Euclidean setting,
SFEMs for the surface Stokes problem based on the standard velocity-pressure formulation
must use compatible discrete spaces. 
Specifically, a discrete inf-sup condition must be satisfied. 
Given that SFEMs 
 utilize the same framework as their Euclidean counterparts, 
employing mappings via affine or polynomial diffeomorphisms, 
one may anticipate that  numerous 
classical inf-sup stable Stokes pairs 
can be adapted to their surface analogues, 
readily enabling the construction of stable SFEMs for \eqref{eqn:Stokes}.

However, the tangential velocity constraint poses a significant hurdle
to constructing stable and convergent SFEMs.  As $\Gamma_h$ is merely Lipschitz continuous,
its outward unit normal is discontinuous at mesh edges and vertices.
As a result, the tangential projection 
of continuous, piecewise smooth functions
does not lead to $\bH^1$-conforming functions.  Moreover,
there do not exist canonical, degree-of-freedom-preserving pullbacks for 
tangential $\bH^1$ vector fields, in particular, the Piola
transform preserves tangentiality and in-plane normal continuity, but not in-plane tangential continuity.
Finally, a continuous, tangential, and piecewise smooth vector field
on $\Gamma_h$ must necessarily vanish on mesh corners except in exceptional cases where all incident triangles are coplanar.
Indeed, at a mesh corner there are at least three
faces emanating from a common vertex, whose outward unit normal vectors
span $\mathbb{R}^3$.  Therefore tangentiality of a continuous vector field 
with respect to each of the three planes implies that it vanishes at the vertex.
Thus any piecewise polynomial space simultaneously satisfying both
tangentiality and continuity exhibits a locking-type phenomenon with
poor approximation properties.


There is a substantial recent literature on numerical approximation of the surface Stokes and related 
problems such as the surface vector Laplace {equation}.  Most of these circumvent the difficulties described above in one of three ways:  by relaxing the pointwise tangential constraint, by relaxing $\bH^1$-conformity of the finite element space, or by using a different formulation of the {surface} Stokes problem.   
For the former, one can weakly impose the tangential constraint via penalization or 
Lagrange multipliers \cite{Fries18,GJOR18,OQRY18,Maxim19,HansboLarsonLarsson20,JORZ21,BJPRV22}.
In principle, this allows one to use inf-sup stable Euclidean Stokes pairs to solve
the analogous surface problem. However, this methodology
requires superfluous degrees of freedom, as the velocity space is approximated by arbitrary vectors in 
$\mathbb{R}^3$ rather than tangential vectors.  In addition an unnatural high-order geometric approximation of the unit normal
of the true surface is needed to obtain optimal-order approximations.  Therefore for problems
in which full information of the exact surface is unknown (e.g., free-boundary problem), these penalization schemes
lead to SFEMs with sub-optimal convergence properties. However, it was shown recently 
in \cite{HP23} that the tangential component of the solution converges optimally for a standard isoparametric
geometry approximation in most cases assuming a correct choice of penalty parameters.  The only exception is the case where tangential $L_2$ errors are considered along with affine (polyhedral) surface approximations.
Alternatively, one may relax $\bH^1$-conformity and use finite element trial and test
functions that are not continuous on the discrete surface $\Gamma_h$. 
In this direction, SFEMs utilizing tangentially- and $\bH({\rm div})$-conforming finite element spaces
such as Raviart-Thomas and Brezzi-Douglas-Marini combined 
with discontinuous Galerkin techniques are proposed and analyzed in  \cite{SurfaceStokes1,SurfaceStokes2}; cf.~\cite{CockburnEtal05} for similar methods for Euclidean Stokes equations.  
Here, additional consistency, symmetry, and stability terms are added to the method. These terms add some complexity to the 
implementation, especially for higher-order surface approximations, but are standard in the context of discontinuous Galerkin methods.
Optimal-order convergence is observed experimentally for a standard SFEM formulation that does not require higher-order approximations of any geometric information.  
Discretizations of stream function formulations of the surface Stokes equations have also appeared in the literature \cite{NVW12,Re20,BR20,BJPRV22}.  However, as with methods weakly enforcing tangentiality, they require higher-order approximation to the surface normal and in addition require computation of curvature information which can in and of itself be a challenging problem. 
 {These methods are} also restricted to simply connected surfaces.  As a final note, {\it trace SFEMs}, in which discretizations of surface PDE are formulated with respect to a background 3D mesh and a corresponding 3D finite element space, are especially important in the context of dynamic surface fluid computations.  Trace formulations are well-developed for $H^1$ conforming/tangentially nonconforming methods and stream function formulations, but have not yet appeared for $\bH({\rm div})$-conforming methods.  

In this paper, we design {a SFEM} for the surface Stokes problem \eqref{eqn:Stokes}
using a strongly tangential finite element space that is based 
on {a} conforming, inf-sup stable Euclidean {pair}. 
The method is based on the standard variational formulation for the Stokes problem
and does not require additional consistency terms or extrinsic penalization.
As far as we are aware, this is the first {SFEM} for the surface Stokes problem
with these properties.
The key issue that we address is the assignment of degrees of freedom (DOFs)
of {tangential} vector fields at Lagrange nodes, in particular,
at vertices of the surface triangulation. 

To expand on this last point and to describe our proposed approach, 
consider a vertex/DOF, call $a$, of the triangulation of the discrete geometry
approximation $\Gamma_h$, and let $\calT_a$ denote the set of faces 
in the triangulation that have $a$ as a vertex.
We wish to interpret and define the values
of tangential vector fields forming our finite element space at this vertex
in a way that ensures the resulting discrete spaces
have desirable approximation and weak-continuity properties.
As the mesh elements in $\calT_h$ generally
lie in different planes, it is immediate that such vector fields
are generally multi-valued at $a$.

Let $\bp =\bp_\gamma$ denote the closest point projection onto $\gamma$,
and note that, because $\gamma$ is smooth, continuous and tangential vector fields
are well-defined and single-valued at $\bp(a)$.  Thus, 
as the Piola transform preserves tangentiality,
a natural assignment is to construct finite element functions $\bv$
with the property $\bv|_{K}(a) = \calP_{\bp^{-1}}\pt{\bv}\big|_{\bp(K)}\ \forall K\in \calT_a$
for some vector field $\pt{\bv}$ tangent at $\bp(a)$, where $\calP_{\bp^{-1}}$ is the Piola transform
of the inverse mapping $\bp^{-1}:\gamma\to \Gamma_h$; see Figure \ref{fig1}.
Imposing this condition on Lagrange finite element DOFs likely
leads to the sought out approximation and weak-continuity properties,
and thus, conceptually may lead to convergent SFEMs for \eqref{eqn:Stokes}.
However, the implementation of the resulting finite element method
requires explicit information about the exact surface $\gamma$ and its closest 
point projection.  Therefore this construction is of little practical value.

Instead of this idealized construction, 
we fix an arbitrary face $K_a\in \calT_a$.
Given the value $\bv|_{K_a}(a)$ and $K\in \calT_a$, we then assign
$\bv|_{K}(a) = \pol_{\bp^{-1}_{K_a}} \bv|_{K_a}(a)$, the Piola transform of  $\bv|_{K_a}(a)$
with respect to the inverse of the closest point projection onto the plane containing $K_a$; see Figure \ref{fig2}.
 This transform is linear with a relatively simple formula (cf.~Definition \ref{def:MaK}), and 
 it  only uses geometric information from $\Gamma_h$.  Moreover,
 we show that this construction is only an $O(h^2)$ perturbation from the idealized setting.
 As a result, the constructed finite element spaces possess sufficient weak continuity properties
to ensure that the resulting scheme is convergent for the surface Stokes problem \eqref{eqn:Stokes}.

To clearly communicate the main ideas and to keep technicalities at minimum, we focus
on a polyhedral approximation to $\gamma$ and on
the lowest-order MINI pair, which in the Euclidean setting
takes the discrete velocity space to be the (vector-valued) linear Lagrange space
enriched with cubic bubbles, and the discrete pressure space to be
the (scalar) Lagrange space. We expect the main ideas to be applicable
to other finite element pairs  (e.g, Taylor--Hood, Scott-Vogelius \cite{ScottVogelius85,JohnEtal17,TaylorHood73,BoffiEtalBook}),
although the stability must be shown on a case-by-case basis.  Below we 
{present numerical experiments demonstrating the viability of our approach for 
$\mathbb{P}^2$ surface approximations paired with a $\mathbb{P}^2-\mathbb{P}^1$ Taylor-Hood finite element space} and plan to address
generalizations of our approach more fully in future works.

The rest of the paper is organized as follows.
In the next section, we introduce the notation and provide
some preliminary results.  In Section \ref{sec-FE},
we define the surface finite element spaces based on the classical
MINI pair.  Here, we  show that the spaces have optimal-order approximation
properties and are inf-sup stable. We also establish weak continuity
properties of the discrete velocity space via an $H^1$-conforming relative
on the true surface.  In Section \ref{sec-FEM}, we define the finite element
space and prove optimal-order estimates in the energy norm. Finally in Section \ref{sec-numerics}
we provide numerical experiments which support the theoretical results.

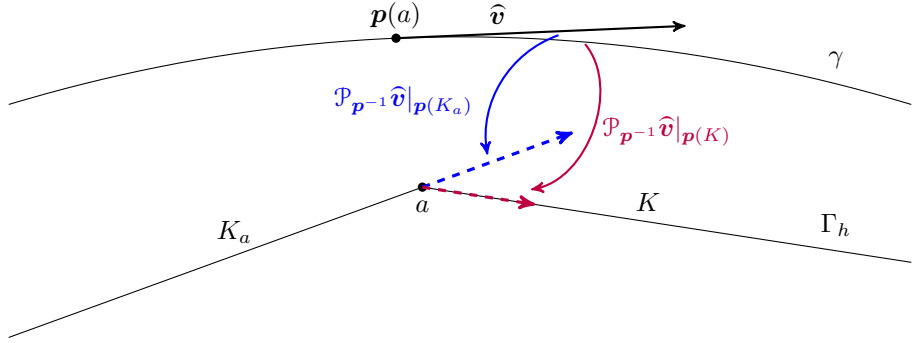
\begin{figure}
\centering
\begin{tikzpicture}

\def\theta{-0.85}
\draw[black, domain=-6:6,smooth,variable=\t]plot (\t,-0.025*\t*\t);
\draw (5,-0.3) node {$\gamma$};  
\node[inner sep = 0pt,minimum size=3.5pt,fill=black!100,circle] (n2) at (\theta,-0.025*\theta*\theta)  {};
 \draw (\theta,-0.025*\theta*\theta+0.3) node {$\bp(a)$};  

\draw[->,thick](\theta,{-0.05*\theta*(\theta-\theta)-0.025*\theta*\theta})--(3,{-0.05*\theta*(3-\theta)-0.025*\theta*\theta});
\draw (0.5,0.3) node {$\pt{\bv}$};

\draw[-](-6,-4)--(-0.5,-2); %
\draw[-](6,-3)--(-0.5,-2);
\draw (5,-2.5) node {$\Gamma_h$};  
\node[inner sep = 0pt,minimum size=3.5pt,fill=black!100,circle] (n3) at (-0.5,-2)  {};
 \draw (-0.5,-2.25) node {$a$};  
\draw (2.5,-2.2) node {$K$};
\draw (-3,-2.6) node {$K_a$};

\draw[->,very thick,dashed,blue](-0.5,-2)--(1.5,{(2/5.5)*(1.5+0.5)-2});
\draw[->,very thick,dashed,purple](-0.5,-2)--(1,{(-1/6.5)*(1+0.5)-2});

\node (phys1) at (.5,{(2/5.5)*(.5+0.5)-2}) {}; 
\node (ref)  at (1.5,{-0.05*\theta*(1.5-\theta)-0.025*\theta*\theta}) {}
edge[pil,bend right=40,blue] (phys1.west);
 
\node (phys2) at (0.75,{(-1/6.5)*(0.75+0.5)-2+0.15}) {}; 
\node (ref)  at (1.5,{-0.05*\theta*(1.5-\theta)-0.025*\theta*\theta}) {}
edge[pil,bend left=60,purple] (phys2.east);
 
\draw[purple] (2.75,-1.25) node {$\calP_{\bp^{-1}}\pt{\bv}|_{\bp(K)}$};
\draw[blue] (-0.75,-0.85) node {$\calP_{\bp^{-1}}\pt{\bv}|_{\bp(K_a)}$};

\end{tikzpicture}
\caption{\label{fig1}A pictorial description of an idealized assignment of nodal values on a one-dimensional surface.
Here, a tangential vector $\pt{\bv}$ at the point $\bp(a)$ is mapped to each element $K$
via the Piola transform with respect to the inverse mapping $\bp^{-1}|_K$.
}
\end{figure}

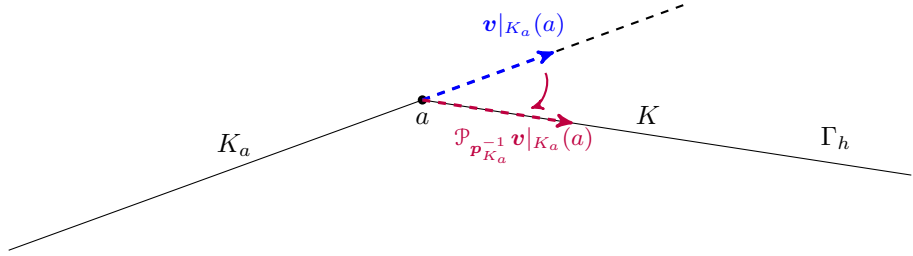
\begin{figure}
\centering
\begin{tikzpicture}

\draw[-](-6,-4)--(-0.5,-2);
\draw[-](6,-3)--(-0.5,-2);
\draw (5,-2.5) node {$\Gamma_h$};  
 \draw (-0.5,-2.25) node {$a$};  
 \node[inner sep = 0pt,minimum size=3.5pt,fill=black!100,circle] (n3) at (-0.5,-2)  {};
 \draw (2.5,-2.2) node {$K$};
\draw (-3,-2.6) node {$K_a$};

\draw[-,dashed,thick](-0.5,-2)--(3,{(2/5.5)*(3+0.5)-2});

\draw[->,very thick,dashed,blue](-0.5,-2)--(1.25,{(2/5.5)*(1.25+0.5)-2});
\draw[->,very thick,dashed,purple](-0.5,-2)--(1.5,{(-1/6.5)*(1.5+0.5)-2});

 \node (phys) at (1,{(-1/6.5)*(1+0.5)-2+0.1}) {};
\node (ref)  at (1.,{(2/5.5)*(1.+0.5)-2}) {}
edge[pil,bend left=50,purple] (phys.west);

\draw[blue] (0.85,-1) node {{\small{$\bv|_{K_a}(a)$}}};
\draw[purple] (0.85,-2.6) node {{\small{$\calP_{\bp_{K_a}^{-1}} \bv|_{K_a}(a)$}}};

\end{tikzpicture}
\caption{\label{fig2}
A pictorial description of our construction on a one-dimensional surface.
The value of a vector field $\bv$ at vertex $a$ restricted to $K_a$
is mapped to $K$ via the Piola transform with respect to the inverse
of the closest point projection onto the plane containing $K_a$.
}
\end{figure}

\section{Notation and Preliminaries}\label{sec-notation}
We assume $\gamma\subset \bbR^3$
is a smooth, connected, and orientable two-dimensional surface
without boundary.  The signed distance function 
of $\gamma$ is denoted by $d$, which
satisfies $d<0$ in the interior of $\gamma$ and $d>0$ in the 
exterior.
We set $\bnu(x) = \nab d(x)$ to be the outward-pointing unit normal 
(where the gradient is understood as a column vector)
and ${\bf H}(x) = D^2 d(x)$ the Weingarten map.
The tangential projection operator
is $\bPi = {\bf I} - \bnu\otimes \bnu$,
where ${\bf I}$ is the $3\times 3$ identity matrix,
and the outer product
of two vectors ${\bm a}$ and ${\bm b}$ satisfies
$({\bm a}\otimes {\bm b})_{i,j} = a_i b_j$.
The smoothness of $\gamma$ ensures
the existence of $\delta>0$ sufficiently small such 
that the closest point projection
\[ 
\bp(x) := x-d(x)\bnu(x)
\]
is well defined in
the tubular region
$U = \{x\in \bbR^3:\ {\rm dist}(x,\gamma)\le \delta\}$.

For a scalar function $q:\gamma\to \bbR$ 
we define its extension
$q^e:U\to \bbR$ via $q^e = q\circ \bp$.
Likewise, for $\bv = (v_1,v_2,v_3)^\intercal :\gamma \to \bbR^3$
its extension $\bv^e:U\to \bbR^3$ satisfies
$(\bv^e)_i = v_i^e$ for $i=1,2,3$.
Define the surface gradient $\nab_\gamma q = \bPi \nab q^e$,
and for a (column) vector field $\bv = (v_1,v_2,v_3)^\intercal :\gamma \to \bbR^3$,
we let $\nab \bv^e = (\nab v_1^e,\nab v_2^e,\nab v_3^e)^\intercal $
denote the Jacobian matrix of $\bv^e$.  
We then see $(\nab \bv^e \bPi)_{i,:} = ((\nab v^e_i)^\intercal  \bPi) = (\bPi \nab v^e_i)^\intercal = (\nabla_\gamma v_i)^\intercal$,
i.e., the $ith$ row of $\nab \bv^e \bPi$ coincides with $(\nabla_\gamma v_i)^\intercal$.
The tangential surface gradient (covariant derivative) of $\bv$ is
defined by $\nab_\gamma \bv = \bPi \nab \bv^e \bPi$,
and the surface divergence operator of $\bv$
is ${\rm div}_\gamma \bv = {\rm tr}(\nab_\gamma \bv)$.
The deformation of a tangential vector field
is defined as the symmetric part of its surface gradient, i.e.,
\[
{\rm Def}_\gamma \bv = \frac12 \big(\nab_\gamma \bv+(\nab_\gamma \bv)^\intercal\big).
\]
For a matrix field ${\bf A}:\bbR^{3\times 3}$, the divergence
${\rm div}_\gamma {\bf A}$ is understood to act row-wise.

Let $L_2(\gamma)$ denote the space of square-integrable
functions on $\gamma$ and let $\mathring{L}_2(\gamma)$
be the subspace of $L_2(\gamma)$ 
consisting of $L_2$-functions with vanishing mean.
We let $W_p^m(\gamma)$ be the Sobolev space
of order $m$ and exponent $p$ on $\gamma$ with corresponding norm $\|\cdot\|_{W^{m}_p(\gamma)}$.
We use the  notation $H^m(\gamma) = W_2^m(\gamma)$ with $\|\cdot\|_{H^m(\gamma)} = \|\cdot \|_{W^m_2(\gamma)}$,
and the convention $|\cdot|_{H^0} = \|\cdot\|_{L_2}$, $|\cdot|_{W^0_p} = \|\cdot\|_{L_p}$.
Analogous vector-valued spaces are denoted in boldface (e.g., $\bL_2(\gamma) = (L_2(\gamma))^3$
and $\bH^1(\gamma) = (H^1(\gamma))^3$).
We let $\bH_T^1(\gamma)$
be the subspace of $\bH^1(\gamma)$ whose members
are tangent to $\gamma$, and set
\[
\bH({\rm div}_\gamma;\gamma) = \{\bv\in \bL_2(\gamma):\ {\rm div}_\gamma \bv\in L_2(\gamma)\}.
\]

Let $\Gamma_h$ be a  polyhedral surface
approximation of $\gamma$ with triangular faces.
We assume that $\Gamma_h$ is an $O(h^2)$ 
approximation in the sense that $d(x) = O(h^2)$
for all $x\in \Gamma_h$. We further assume
 $h$ is sufficiently small to ensure $\Gamma_h\subset U$, in particular,
the closest point projection is well-defined on $\Gamma_h$.
We denote by $\calT_h$ {the} set of faces of $\Gamma_h$,
and assume this triangulation is shape-regular (i.e.,
the ratio of the diameters of the inscribed and circumscribed 
circles of each face is uniformly bounded). 
For simplicity and to ease the presentation,
we further assume that $\calT_h$ is quasi-uniform, i.e., 
 $h:=\max_{K'} {\rm diam}(K') \approx {\rm diam}(K)$
 for all $K\in \calT_h$.  
  The image of the mesh elements and the resulting set on the exact surface are given, respectively, by 
\[
K^\gamma = \bp(K),\qquad \calT_h^\gamma = \{\bp( K):\  K\in  \calT_h\}. 
\]
  
We use the notation $a\lesssim b$ (resp., $a\gtrsim b$) if
there exists a constant $C>0$ independent of the mesh parameter $h$
such that $a\le C b$ (resp., $a\ge C b$).  {The statement
$a\approx b$ means $a\lesssim b$ and $a\gtrsim b$.}

 Set $\calV_h$ to be the set of vertices in $\calT_h$,
 and for each $K\in \calT_h$, let $\calV_K$ denote
 the set of three vertices of $K$.
 For each $a\in \calV_h$, let $\calT_a\subset \calT_h$
 denote the set of faces having $a$ as a vertex.
 For $K\in \calT_h$, we define
 the patches
 \[
 \omega_K = \mathop{\bigcup_{K'\in \calT_h}}_{\bar K'\cap \bar K\neq \emptyset} K',\qquad
 \omega_K' = \mathop{\bigcup_{K'\in \calT_h}}_{\bar K'\cap \overline \omega_K\neq \emptyset} K',
 \]
so that $\omega_K\subset \omega_K' \subset \Gamma_h$. 
 The patches $\omega_{K^\gamma}$ and $\omega'_{K^\gamma}$
 associated with $K^\gamma = \bp(K)$ are defined analogously.

 The (piecewise constant) outward unit normal of $\Gamma_h$
 is denoted by $\bnu_h$,
 and we shall use the notation $\bnu_K = \bnu_h|_K\in \bbR^3$,
 its restriction to $K\in \calT_h$.  
We assume that  $|\bnu \circ \bp - \bnu_h|\lesssim h$. 
 The tangential projection
 with respect to $\Gamma_h$ is $\bPi_h = {\bf I}-\bnu_h\otimes \bnu_h$,
 and we assume there exists $c>0$ independent of $h$ such that 
$\bnu \cdot \bnu_h\ge c>0$ on $\Gamma_h$.
We let $\mu_h(x)$ satisfy $\mu_h d\sigma_h(x) = d\sigma(\bp(x))$,
where $d\sigma$ and $d\sigma_h$ are surface measures
of $\gamma$ and $\Gamma_h$, respectively.  In particular,
\[
\int_{\Gamma_h} (q\circ \bp)\mu_h = \int_\gamma q\qquad \forall q\in {L_1(\gamma)}.
\]
From  \cite[Proposition 2.5]{Demlow09},
we have 
\begin{equation}\label{eqn:muForm}
\mu_h(x) = \bnu(x)\cdot \bnu_h(x) \prod_{i=1}^2 (1-d(x)\kappa_i(x))\qquad x\in \Gamma_h,
\end{equation}
and  
\begin{equation}\label{eqn:muBound}
|1-\mu_h(x)|\lesssim h^2,
\end{equation}
where $\{\kappa_1,\kappa_2\}$ are the eigenvalues
of ${\bf H}$, whose corresponding eigenvectors are orthogonal to $\bnu$.
We set $\mu_K = \mu_h|_K$ to be the restriction of $\mu_h$ to $K\in \calT_h$.

Surface differential operators with respect to $\Gamma_h$
are denoted and defined analogously to those on $\gamma$.
We also set  ($m\in \mathbb{N}$)
\[
\bH^m_h(\Gamma_h) = \{\bv\in \bL_2(\Gamma_h):\ \bv|_K\in \bH^m(K)\ \forall K\in \calT_h\},\qquad \|\bv\|_{H^m_h(K)}^2 = \sum_{K\in \calT_h} \|\bv\|_{H^m(K)}^2
\]
to be the piecewise $H^m$ Sobolev space and norm, respectively.
Likewise,  $\bH^m_h(\gamma)$ is the piecewise
Sobolev space with respect to $\calT_h^\gamma$ with
corresponding norm $\|\bv\|_{H^m_h(\gamma)}^2 = \sum_{K\in \calT_h} \|\bv\|_{H^m(K^\gamma)}^2$,
and ${\rm Def}_{\gamma,h}$ denotes the piecewise deformation operator with respect to $\calT_h^\gamma$.

We end this section by stating a well-known characterization of 
$\bH({\rm div}_{\Gamma_h};\Gamma_h) =\{\bv\in \bL_2(\Gamma_h):\ {\rm div}_{\Gamma_h} \bv \in L_2(\Gamma_h)\}$.
For each edge $e$ of the mesh, 
denote by $K_1^e,K_2^e\in \calT_h$
the two triangles in the mesh such that $e = \p K_1^e\cap \p K_2^e$.
Let $\bn_j^e$ denote the outward in-plane normal to {$\p K_j^e$},
and note that in general, $\bn^e_1\neq -\bn^e_2$ on $e$.
Then a vector field $\bv\in \bH^1_h(\Gamma_h)$ satisfies
$\bv\in \bH({\rm div}_{\Gamma_h};\Gamma_h)$ if and only if \cite{SurfaceStokes2}
\begin{equation}\label{eqn:HdivProp}
\bv_1\cdot \bn_1^e|_e + \bv_2\cdot \bn_2^e|_e = 0\quad \text{for all edges }e,
\end{equation}
where $\bv_j = \bv|_{K_j^e}$.

\subsection{Extensions and lifts}
For the rest of the paper, 
we view the closest point projection
as a mapping from the discrete surface approximation
to the true surface, i.e.,
$\bp:\Gamma_h\to \gamma$.
Restricted to $\Gamma_h$
the projection is a bijection, 
and in particular has a well-defined
inverse: $\bp^{-1}:\gamma \to \Gamma_h$.
Recall that for a scalar or vector-valued function $q$ on $\gamma$, 
its extension (now to $\Gamma_h$) is $q^e = q\circ \bp$. 
For a scalar or vector-valued function $q$ defined on $\Gamma_h$, we define its lift
via $q^\ell = q\circ \bp^{-1}$. Note that $(q^\ell)^e = q$ on $\gamma$ and likewise, $(q^e)^\ell = q$ on $\Gamma_h$.
For $q\in H^m_h(\gamma)\ (m=0,1,2)$,
there holds
\begin{equation}\label{eqn:extEquivalence}
\|q\|_{H^m(K^\gamma)}\approx \| q^e\|_{H^m(K)}\qquad \forall K\in \calT_h,
\end{equation}
which follows 
from a change of variables, the chain rule, and the smoothness assumptions
of $\gamma$ (cf.~\cite{Dziuk88}).

\subsection{Surface Piola transforms}
Following \cite{Steinmann08,CD16,SurfaceStokes1} we summarize the divergence-conforming 
Piola transform with respect to a mapping between surfaces.
Let $\calS_0$ and $\calS_1$ be two sufficiently smooth surfaces,
and let $\Phi:\calS_0\to \calS_1$ be a diffeomorphism with inverse
$\Phi^{-1}:\calS_1\to \calS_0$.  Let $d\sigma_i$ be
the surface measure of $\calS_i$,
and let $\mu$ formally satisfy $\mu d\sigma_0 = d\sigma_1$.
Then the Piola transform of a vector field $\bv:\calS_0\to \bbR^3$
with respect to $\Phi$
is given by 
\[
(\calP_\Phi \bv)\circ \Phi= \mu^{-1} D\Phi \bv.
\]
Likewise, for $\bv:\calS_1\to \bbR^3$ its Piola transform
with respect to $\Phi^{-1}$ is
\[
(\calP_{\Phi^{-1} \bv)\circ \Phi^{-1} = (\mu\circ \Phi^{-1}) D \Phi^{-1} \bv.}
\]
Similar to the Euclidean setting, there holds
\begin{equation}\label{eqn:DivPhi}
{\rm div}_{\calS_0} \bv = \mu {\rm div}_{\calS_1} \calP_\Phi \bv\qquad \forall \bv\in \bH({\rm div};\calS_0),
\end{equation}
in particular, $\calP_{\Phi}:\bH({\rm div}_{\calS_0};\calS_0)\to \bH({\rm div}_{\calS_1};\calS_1)$
and $\calP_{\Phi^{-1}}:\bH({\rm div}_{\calS_1};\calS_1)\to \bH({\rm div}_{\calS_0};\calS_0)$
are bounded mappings.  Moreover, as $D\Phi$ and $D\Phi^{-1}$
are tangent maps, the Piola transform yields tangential vector fields:
if $\bnu_j$ is the unit normal of the surface
$\calS_j$,
then $(\calP_\Phi \bv)\cdot \bnu_1=0$ on $\calS_1$
and $(\calP_{\Phi^{-1}} \bv)\cdot \bnu_0=0$ on $\calS_0$.

In the case $\Phi = \bp$, $\calS_0 = \Gamma_h$,
and $\calS_1 = \gamma$ (so that $\mu=\mu_h$),
the Piola transform of $\bv:\Gamma_h\to \bbR^3$ with $\bPi_h \bv= \bv$
is  \cite{CD16,SurfaceStokes1}
\begin{equation}\label{eqn:Piola}
\pt \bv \circ \bp := \calP_{\bp}  \bv =  \frac1{\mu_h} \big[\bPi - d {\bf H}\big]  \bv,
\end{equation}
whereas the Piola transform of $\bv:\gamma\to \bbR^3$
with respect to the inverse $\bp^{-1}$ is given by
\begin{equation}\label{eqn:invPiola}
\ipt \bv := \calP_{\bp^{-1}}\bv = \mu_h \Big[{\bf I}- \frac{\bnu\otimes \bnu_h}{\bnu\cdot \bnu_h}\Big][{\bf I}-d {\bf H}]^{-1} (\bv \circ \bp). 
\end{equation}
Note that $\pt{\ipt{\bv}} = \bv$ on $\gamma$ and $\ipt{\pt{\bv}} = \bv$ on $\Gamma_h$.
Moreover, it follows from \eqref{eqn:DivPhi} that for all $K\in \calT_h$,
\begin{subequations}
\label{eqn:PiolaInv1}
\begin{align}
\int_{K^\gamma} ({\rm div}_\gamma \bv) q &= \int_K ({\rm div}_{\Gamma_h} \ipt \bv)  q^e\qquad \forall \bv\in \bH({\rm div}_\gamma;K^\gamma),\ q\in {L_2(K^\gamma)}.
\end{align}
and 
\begin{align}
\int_K ({\rm div}_{\Gamma_h} \bv)  q & = \int_{K^\gamma} ({\rm div}_\gamma \pt \bv) q^\ell \qquad \forall \bv\in \bH({\rm div}_{\Gamma_h};K),\ q\in {L_2(K)}.
\end{align}
\end{subequations}

The following lemma states
 the equivalence of norms of vector fields and their Piola transforms.
%
 \begin{lemma}\label{lem:PiolaNormEquiv}
For $K\in \calT_h$, let $\bv:K^\gamma\to \bbR^3$
and $\ipt{\bv} = \calP_{\bp^{-1}} \bv:K\to \bbR^3$
be related by \eqref{eqn:invPiola} restricted to $K$. 
Then if $\bv\in \bH^m(K^\gamma)$ for some $m\in \{0,1,2\}$, 
there holds $\ipt{\bv}\in H^m(K)$.  Moreover,
%
\begin{equation}\label{eqn:PiolaNormEquiv}
\|\bv\|_{H^m(K^\gamma)}  \approx \|\ipt \bv\|_{H^m(K)}.
\end{equation}
 \end{lemma}
 \begin{proof}
The proof for the cases $m=0,1$ is found in \cite[Lemma 4.1]{SurfaceStokes1}.
The case $m=2$ follows from similar arguments and is therefore omitted.
 \end{proof}

We also need a similar result
that relates the ${L_2}$ norm
of the deformation tensors
of $\bv$ and its Piola transform $\pt{\bv}$.
The proof of the following result
is given in the appendix.
\begin{lemma}\label{lem:DefRelation}
For $K\in \calT_h$, let $\bv\in \bH_T^1(K)$ and $\pt \bv\in \bH_T^1(K^\gamma)$
be related via $\pt \bv = \calP_{\bp} \bv$.
Then
\begin{align}\label{eqn:DefDef}
|{\rm Def}_{\gamma} \pt \bv-({\rm Def}_{\Gamma_h} \bv)\circ \bp^{-1}|\lesssim  h \big(|(\nab_{\Gamma_h} \bv)\circ \bp^{-1}|+ |\bv\circ \bp^{-1}|\big).
\end{align}
Consequently, by a change of variables,
\begin{align}
\label{eqn:DefRelation}
\|{\rm Def}_{\gamma} \pt \bv\|_{L_2(K^\gamma)}\lesssim \|{\rm Def}_{\Gamma_h} \bv\|_{L_2(K)}+ h \big( \|\nab_{\Gamma_h} \bv\|_{L_2(K)}
+ \|\bv\|_{L_2(K)}\big).
\end{align}
\end{lemma}

We now apply the above definitions of Piola transforms to mappings between planes (surface triangles), which is critical to our construction of vertex degrees of freedom for vector fields on $\Gamma_h$.
\begin{definition}\label{def:MaK}
For each vertex $a\in \calV_h$ in the triangulation, 
we arbitrarily choose a single (fixed) face $K_a\in \calT_a$.
For $K\in \calT_a$, we define
$\calM_a^K:\bbR^3\to \bbR^3$ by
%
%
\begin{equation}\label{eqn:calMDef}
\calM_a^K \bx= \Big(\bnu_{K_a} \cdot \bnu_K \Big[ {\bf I} - \frac{\bnu_{K_a} \otimes \bnu_{K}}{\bnu_{K_a}\cdot \bnu_{K}}\Big]\Big) \bx,
\end{equation}
where we recall $\bnu_{K_a}$ and $\bnu_{K}$ are the outward unit normals of $K_a$ and $K$, respectively.
In particular, $\calM_a^K \bx$ is the Piola transform of $\bx$ with respect
to the inverse of the closest point projection onto the plane containing $K_a$ (cf.~\eqref{eqn:invPiola}).
\end{definition}
\begin{remark}
By properties of the Piola transform, $\calM_a^K \bx$
is tangential to $K$, i.e., $(\calM_a^K \bx)\cdot \bnu_K=0$ for all $x\in \mathbb{R}^3$.
\end{remark}

We next show that the ``ideal'' and ``practical'' interpretations of vectors at vertices discussed in the introduction (cf. Figures \ref{fig1} and \ref{fig2}) do not differ by too much.
\begin{lemma}\label{lem:PiolaD} 
Fix $a\in \calV_h$ and let $\bu$ lie in the tangent plane of $\gamma$ at 
$\bp(a)$.  For $K\in \calT_a$, let $\ipt{\bu}_K = \calP_{\bp^{-1}} \bu|_K$ be the Piola transform of $\bu$ to $K$ via the inverse of the closest point projection (cf.~\eqref{eqn:invPiola}).
Then 
\begin{align}
\label{piola_defect}
|\ipt{\bu}_K-\mathcal{M}_a^K \ipt \bu_{K_a}|
\lesssim h^2 |\ipt \bu_{K_a}| \le h^2 |\bu|.
\end{align} 
\end{lemma}
\begin{proof}
Using ${\bf H} \bnu = \frac12 \nab |\bnu|^2 = 0$,
we have ${\bf H}\bPi = {\bf H}$.  
Note in addition that $\left [ {\bf I} - \frac{\bnu \otimes \bnu_K}{\bnu \cdot \bnu_K} \right ] \bPi = \left [ {\bf I} - \frac{\bnu \otimes \bnu_K}{\bnu \cdot \bnu_K} \right ]$.  
Therefore by  \eqref{eqn:invPiola},  \eqref{eqn:Piola}, and \eqref{eqn:muForm} we have
\begin{align*}
\begin{aligned} 
\ipt{\bu}_K 
& = \mu_K \left [ {\bf I} - \frac{\bnu \otimes \bnu_K}{\bnu \cdot \bnu_K} \right ] \left [ {\bf I} - d {\bf H} \right ] ^{-1} \bu\\
& = \mu_K \left [ {\bf I} - \frac{\bnu \otimes \bnu_K}{\bnu \cdot \bnu_K} \right ] \left [ {\bf I} - d {\bf H} \right ] ^{-1} \frac{1}{\mu_{K_a}} \left [ \bPi - d {\bf H} \right ] \ipt{\bu}_{K_a}\\
& = \mu_K \left [ {\bf I} - \frac{\bnu \otimes \bnu_K}{\bnu \cdot \bnu_K} \right ] \left [ {\bf I} - d {\bf H} \right ] ^{-1} \frac{1}{\mu_{K_a}} \left [ {\bf I}- d {\bf H} \right ] \bPi \ipt{\bu}_{K_a}
\\ & = \frac{\bnu \cdot \bnu_K}{\bnu \cdot \bnu_{K_a}} \left  [ {\bf I} - \frac{\bnu \otimes \bnu_K}{\bnu \cdot \bnu_K} \right ] {\bf \Pi} \ipt{\bu}_{K_a}
\\ & = \frac{\bnu \cdot \bnu_K}{\bnu \cdot \bnu_{K_a}} \left  [ {\bf I} - \frac{\bnu \otimes \bnu_K}{\bnu \cdot \bnu_K} \right ] \ipt{\bu}_{K_a}.
\end{aligned}
\end{align*}

Now $|\bnu-\bnu_K| + |\bnu_{K_a}-\bnu_K| \lesssim h$, $|1-\bnu\cdot \bnu_K| = \frac{1}{2} |\bnu-\bnu_K|^2 \lesssim h^2$, and $|1-\bnu_{K_a} \cdot \bnu_K|\lesssim h^2$.   
Thus employing \eqref{eqn:calMDef} and the identity $\bnu_{K_a} \cdot \ipt{\bu}_{K_a}=0$ we have
\begin{align*}
\begin{aligned}
|\ipt{\bu}_K-\calM_a^K \ipt{\bu}_{K_a}| 
& =  \left | \left (  \frac{\bnu \cdot \bnu_K}{\bnu \cdot \bnu_{K_a}} \left [ {\bf I} - \frac{\bnu \otimes \bnu_K}{\bnu \cdot \bnu_K} \right ] - \bnu_{K_a} \cdot \bnu_K \left [ {\bf I} - \frac{ \bnu_{K_a} \otimes \bnu_K}{\bnu_{K_a} \cdot \bnu_K} \right ] \right )  \ipt{\bu}_{K_a} \right |\\
& \lesssim  \left |   \big  [  \left [ {\bf I} - \bnu \otimes \bnu_K \right ] - \left [ {\bf I} - \bnu_{K_a} \otimes \bnu_K \right ]   \big ]  \ipt{\bu}_{K_a} \right |+h^2 |\ipt{\bu}_{K_a}|
\\ & = |(\bnu -\bnu_{K_a}) \otimes \bnu_K \ipt{\bu}_{K_a}| + h^2 |\ipt{\bu}_{K_a}|
\\ & = |(\bnu-\bnu_{K_a}) \otimes (\bnu_K-\bnu_{K_a}) \ipt{\bu}_{K_a}| + h^2 |\ipt{\bu}_{K_a}|
\\ & \lesssim h^2 |\ipt{\bu}_{K_a}|. 
\end{aligned} 
\end{align*}
Finally noting that $|\ipt{\bu}_{K_a}| = | \calP_{\bp|_{K_a}^{-1}} \bu| \lesssim {|\bu|}$ (cf.~\eqref{eqn:invPiola}) completes the proof.
\end{proof}

\begin{remark}  In {SFEMs} it is common to use a higher-order surface approximation 
$\Gamma$ to $\gamma$ of polynomial degree $k$  (here we consider $k=1$).  In that case $\nu_K$ is no 
longer constant on $K$, and we have $|\nu(a)-\nu_K(a)|\lesssim h^k$.  The results of Lemma \ref{lem:PiolaD}
 easily generalize to this situation with $h^{2k}$ replacing $h^2$ on the right hand side of \eqref{piola_defect}.
\end{remark}

\section{Finite element spaces and inf-sup stability}\label{sec-FE}
By utilizing Lemma \ref{lem:PiolaD}, we can construct tangential 
finite element spaces on the surface approximation $\Gamma_h$ using nodal (Lagrange) basis functions. 
The essential idea is to enforce continuity at nodal degrees of freedom in a weak sense through 
the mapping $\calM_a^K$ given in Definition \ref{def:MaK}. Although this procedure 
does not yield a globally continuous finite element space, it preserves in-plane normal continuity and exhibits 
 weak continuity properties. These properties are generally sufficient for achieving convergence in second-order elliptic problems.

In the following discussion, we focus on the construction of the lowest-order MINI Stokes pair for simplicity \cite{mini}. 
However, we expect that Definition \ref{def:MaK} and Lemma \ref{lem:PiolaD} provide a general framework 
for constructing convergent finite element schemes based on classical and conforming finite element pairs such as Taylor-Hood and Scott-Vogelius \cite{BoffiEtalBook}.

\subsection{Surface MINI space and approximation properties}

Let $\hat K$ be the reference triangle with vertices $(0,0),(1,0),(0,1)$,
and for $K\in \calT_h$, let $F_K:\hat K\to K$ be
an affine diffeomorphism.  
The constant Jacobian matrix of $F_K$ is denoted
by $DF_K\in \bbR^{3\times 2}$.  Note that the columns
of $DF_K$ span the tangential space of $K$.
For a vector-valued function $\hat \bv:\hat K\to \bbR^2$,
its Piola transform with respect  to $F_K$ is given by
\begin{align}
\label{ref_piola}
\bv(x) = (\calP_{F_K} \hat \bv)(x):= \frac1{J}DF_K \hat \bv(\hat x),\qquad x = F_K(\hat x),
\end{align}
and $J = \sqrt{\det(DF_K^\intercal DF_K)}$.

Let $b_{K}$ be the standard cubic bubble function on $K$,
i.e., the product of the three barycentric coordinates of $K$.
The local MINI space defined on the reference triangle
is given by $\hat \bV:=\bpol_1(\hat K) \oplus b_{\hat K} \bpol_0(\hat K)$,
where $\pol_k(D)$ is the space of polynomials of degree $\le k$
with domain $D$, and {$\bpol_k(D) = [\pol_k(D)]^2$}.
We then define the surface finite element spaces 
on $\Gamma_h$ as
\begin{align*}
\bV_h &= \{\bv\in \bL_2(\Gamma_h):\forall K\in \calT_h \, \exists \hat \bv\in \hat \bV,\ \bv_K = \calP_{F_K}\hat \bv\  ;\
\bv_K(a) = \calM_a^K(\bv_{K_a}(a))\ \forall K\in \calT_a,\ \forall a\in \calV_h\},\\
Q_h & = \{q\in H^1(\Gamma_h)\cap \mathring{L}_2(\Gamma_h):\ q_K\in \pol_1(K)\ \forall K\in \calT_h\},
\end{align*}
where $\bv_K = \bv|_K$ is the restriction of $\bv$ to $K\in \calT_h$,
 $\calM_a^K$ is defined in Definition \ref{def:MaK}, and we recall $\calV_h$
 is the set of vertices in $\calT_h$.

For $\bv\in \bV_h$, we let $\bv_L$ denote the linear portion of $\bv$, i.e.,
$\bv_L$ is the unique tangential and piecewise linear vector
in $\bV_h$ satisfying $(\bv_L)_{K_a}(a) = \bv_{K_a}(a)$ for all $a\in \calV_h$.
We then have the following identity on each $K\in \calT_h$:
\begin{equation}\label{eqn:MiniDecomp}
\bv = \bv_L + 60 b_K \fint_K (\bv - \bv_L)\qquad \forall \bv\in \bV_h,
\end{equation}
where 
 we have used the fact $\int_K b_K = |K|/60$ for all $K\in \calT_h$.

Because the columns of
$DF_K$ span the tangential space of $K$,
we see that functions in the discrete velocity space $\bV_h$ are tangential, i.e., $\bv\cdot \bnu_h = 0$
for all $\bv\in \bV_h$.  In addition, due to the normal-preserving properties
of the transform $\calM_a^K$, the space is $H({\rm div})$-conforming
as the next result shows.

\begin{proposition}\label{prop:Hdiv}
There holds $\bV_h \subset \bH({\rm div}_{\Gamma_h};\Gamma_h)$.
\end{proposition}
\begin{proof}
Due to the properties of the cubic bubble, it is sufficient
to show that the linear component of $\bv\in \bV_h$ 
satisfies the in-plane normal continuity condition \eqref{eqn:HdivProp}
across all edges in $\calT_h$.

Let $a\in \calV_h$ be a vertex of $\calT_h$,
and let $K_1,K_2\in \calT_a$ be two elements that have $a$
as a vertex and share a common edge $e = \p K_1\cap \p K_2$.
Denote by $\bn^e_j$ the  in-plane outward unit normal
vector with respect to $\p K_j$ restricted to $e$.

Using the definitions of the finite element space and
the operator $\calM_a^K$, along with the Binet-Cauchy identity,
there holds for any $\bv\in \bV_h$,
\begin{align*}
\bv_j(a) \cdot \bn_j^e
&  = \calM_a^{K_j}(\bv_{K_a}(a)) \cdot \bn^e_j\\
& = (\bnu_{K_a}\cdot \bnu_{K_j})(\bv_{K_a}(a) \cdot \bn^e_j) - (\bnu_{K_a}\cdot \bn^e_j)(\bnu_{K_j}  \cdot \bv_{K_a}(a)) =
 (\bnu_{K_a}\times \bv_{K_a}(a))\cdot (\bnu_{K_j} \times \bn^e_j),
\end{align*}
where $\bv_j = \bv_{K_j} = \bv|_{K_j}$.
Therefore,
\begin{align*}
\bv_1(a) \cdot \bn^e_1+\bv_2(a)\cdot \bn^e_2
& =  (\bnu_{K_a}\times \bv_{K_a}(a))\cdot \big((\bnu_{K_1} \times \bn^e_1) +(\bnu_{K_2} \times \bn^e_2)\big)=0.
\end{align*}
Because $\bv_j$ is a linear polynomial on $e$,
we conclude that \eqref{eqn:HdivProp} is satisfied on all edges.
This implies the desired result $\bV_h \subset \bH({\rm div}_{\Gamma_h};\Gamma_h)$.
\end{proof}

\begin{lemma}\label{lem:Interp}
For each $\bw\in \bC(\gamma)\cap \bH^1_T(\gamma)\cap \bH_h^2(\gamma)$,
there exists $\bI_h \ipt{\bw}\in \bV_h$ such that
\begin{align*}
h^m \|\ipt{\bw}-\bI_h\ipt{\bw}\|_{H^m(K)}\lesssim h^2 \|\bw\|_{H^2_h(\omega_{K^\gamma})}\qquad \forall K\in \calT_h,\quad m=0,1,
\end{align*}
with $\ipt{\bw} = \calP_{\bp^{-1}} \bw$.
\end{lemma}
\begin{proof}
Given  $\bw\in  \bC(\gamma)\cap \bH^1_T(\gamma)\cap \bH_h^2(\gamma)$,
we uniquely define $\bv:=\bI_h \ipt{\bw}\in \bV_h$ such that each
component of $\bv$ is a piecewise linear polynomial
and satisfies
\begin{equation}\label{eqn:InterpvhDOFs}
\begin{aligned}
\bv_{K_a}(a) &= \ipt \bw_{K_a}(a)\qquad &&\forall a\in \calV_h.
\end{aligned}
\end{equation}

Let $\ipt{\bw}_I$
be the elementwise (discontinuous), linear interpolant of $\ipt{\bw}$ with respect to the vertices in $\calT_h$, i.e.,
$(\ipt{\bw}_I)_K\in \bpol_1(K)$ with $(\ipt{\bw}_I)_K(a) = \ipt{\bw}_K(a)$ for all $K\in \calT_h$ and $a\in \calV_K$.
By Lemma \ref{lem:PiolaD} (with $\bu = \bw(a)$), \eqref{eqn:InterpvhDOFs},
and the definition of $\bV_h$ we have for each vertex $a\in \calV_h$,
\begin{align*}
|(\ipt{\bw}_I - \bv)_K(a)| = |(\ipt{\bw}_K - \calM_a^K \bv_{K_a})(a)| =  |(\ipt{\bw}_K - \calM_a^K \ipt{\bw}_{K_a})(a)| \lesssim 
h^2 |\ipt{\bw}_{K_a}(a)|
\qquad \forall K\in \calT_a.
\end{align*}
Consequently, by standard inverse estimates,
\begin{align*}
h^m \|\ipt{\bw}_I- \bv\|_{H^m(K)}
&\lesssim  \|\ipt{\bw}_I - \bv\|_{L_2(K)}\\
&\lesssim h\|\ipt{\bw}_I-\bv\|_{L_\infty(K)} = 
h\max_{a\in \calV_K} |(\ipt{\bw}_I - \bv)_K(a)|\lesssim h^3 \max_{a\in \calV_K}  |\ipt{\bw}_{K_a}(a)|\quad m=0,1.
\end{align*}
Using inverse estimates once again, 
and applying standard interpolation results yields
\begin{align*}
\max_{a\in \calV_K}  |\ipt{\bw}_{K_a}(a)|
&\le \|\ipt{\bw}_I\|_{L_\infty(\omega_K)}\\
&\lesssim h^{-1}\|\ipt{\bw}_I\|_{L_2(\omega_K)}\\
&\lesssim h^{-1} (\|\ipt\bw\|_{L_2(\omega_K)}+\|\ipt{\bw} - \ipt{\bw}_I\|_{L_2(\omega_K)})\\
&\lesssim h^{-1} (\|\ipt\bw\|_{L_2(\omega_K)}+ h^2\ |\ipt{\bw}|_{H^2_h(\omega_K)}).
\end{align*}
Therefore,
\begin{align*}
h^m \|\ipt{\bw}_I - \bv\|_{H^m(K)}\lesssim  h^2 \|\ipt{\bw}\|_{H^2_h(\omega_K)}\qquad m=0,1,
\end{align*}
and so by \eqref{eqn:PiolaNormEquiv},
\begin{align*}
h^m \|\ipt{\bw}-{\bv}\|_{H^m(K)}
&\lesssim h^m \|\ipt{\bw}-\ipt{\bw}_I\|_{H^m(K)}+h^m \|\ipt{\bw}_I-\bv\|_{H^m(K)}\lesssim h^2 \|\bw\|_{H^2_h(\omega_{K^\gamma})}.
\end{align*}
\end{proof}

\begin{lemma}\label{lem:InfSup}
There exists a constant $\alpha>0$ independent
of $h$ such that
\begin{equation}\label{eqn:dInfSup}
\sup_{\bv\in \bV_h\backslash \{0\}} \frac{\int_{\Gamma_h} ({\rm div}_{\Gamma_h} \bv)q}{\|\bv\|_{H^1_h(\Gamma_h)}}\ge \alpha \|q\|_{L_2(\Gamma_h)}\qquad \forall q\in Q_h.
\end{equation}
\end{lemma}
\begin{proof}
Fix $ q\in Q_h$, and 
let $\bw\in \bH^1_T(\gamma)$ satisfy \cite{JankuhnEtal18,DziukElliott13,Demlow09}
\begin{equation}\label{eqn:wDef}
{\rm div}_\gamma \bw = (\mu_h^{-1}q)^\ell\in \mathring{L}_2(\gamma),\quad\text{and}\quad
\|\bw\|_{H^1(\gamma)}\lesssim \|(\mu_h^{-1}q)^\ell\|_{L_2(\gamma)}\lesssim \|q\|_{L_2(\Gamma_h)},
\end{equation}
where we used \eqref{eqn:extEquivalence} and \eqref{eqn:muBound} 
in the last step.  Let $\bI^{SZ}_h \bw^e$ 
be the Scott-Zhang interpolant of the extension $\bw^e\in \bH^1(\Gamma_h)$
onto the space of continuous piecewise linear polynomials with respect to $\calT_h$
\cite{CamachoDemlow15,DemlowDziuk07}, and set  $\bw_h = \bPi (\bI^{SZ}_h \bw^e)^\ell \in \bH^1_T(\gamma)$.
From \eqref{eqn:extEquivalence}, $\bw-\bw_h = \Pi (\bw-\bI^{SZ}_h \bw^e)$, and approximation properties of the Scott-Zhang interpolant,
there holds on each $K\in \calT_h$,
\begin{equation}\label{eqn:SZL2}
\begin{split}
\|\bw_h\|_{H^1(K^\gamma)}+h^{-1} \|\bw - \bw_h\|_{L_2(K^\gamma)}
&\lesssim
\|(\bI^{SZ}_h \bw^e)^\ell\|_{H^1(K^\gamma)}+h^{-1} \|\bw - (\bI^{SZ}_h \bw^e)^\ell\|_{L_2(K^\gamma)} \\
&\lesssim
\|\bI^{SZ}_h \bw^e\|_{H^1(K)}+h^{-1} \|\bw^e - \bI^{SZ}_h \bw^e\|_{L_2(K)} \\
&\lesssim \|\bw^e\|_{H^1(\omega_K)}\lesssim  \|\bw\|_{H^1(\omega_{K^\gamma})}.
\end{split}
\end{equation}

Noting $\bw_h\in \bC(\gamma)\cap \bH^1_T(\gamma)\cap \bH_h^2(\gamma)$,
we define $\bv\in\bV_h$ such that
\[
\bv_K = \big({\bI_h \ipt{\bw}_h}\big)_K + 60 b_K \fint_K (\ipt{\bw}-{\bI_h \ipt{\bw}_h})\qquad \forall K\in \calT_h,
\]
where $\bI_h \ipt{\bw}_h$ is given in Lemma \ref{lem:Interp}.
We then have $\int_K \bv = \int_K \ipt{\bw}$, and by \eqref{eqn:SZL2} and Lemma \ref{lem:Interp},
\begin{align*}
\|\bv\|_{H^1(K)}
&\le \|\bI_h \ipt{\bw}_h\|_{H^1(K)}
+\|\bv-{\bI_h \ipt{\bw}_h}\|_{H^1(K)}\\
&\lesssim \|\ipt{\bw}\|_{H^1(K)}+\|\ipt{\bw}-\bI_h \ipt{\bw}_h\|_{H^1(K)}+ h^{-1}\|\ipt{\bw}-{\bI_h \ipt{\bw}_h}\|_{L_2(K)}\\
&\le \|\ipt{\bw}\|_{H^1(K)}+\|\ipt{\bw}-\ipt{\bw}_h\|_{H^1(K)}+\|\ipt{\bw}_h-{\bI_h \ipt{\bw}_h}\|_{H^1(K)}
\\ & ~~~~~~~+h^{-1} (\|\ipt{\bw}-\ipt{\bw}_h\|_{L_2(K)}+\|\ipt{\bw}_h-{\bI_h \ipt{\bw}_h}\|_{L_2(K)})\\
&\lesssim \|\bw\|_{H^1(\omega_{K^\gamma})} + h\|\bw_h\|_{H^2_h(\omega_{K^\gamma})}.
%
\end{align*}
By \eqref{eqn:extEquivalence}, a standard inverse estimate, and the $H^1$-stability
properties of the Scott-Zhang interpolant,
\[
h\|\bw_h\|_{H^2_h(\omega_{K^\gamma})}\lesssim h \|(\bI^{SZ}_h \bw^e)^\ell\|_{H^2_h(\omega_{K^\gamma})}\lesssim h\|\bI^{SZ}_h \bw^e\|_{H^2_h(\omega_K)}
\lesssim \|\bI^{SZ}_h\bw^e\|_{H^1(\omega_K)}\lesssim \|\bw\|_{H^1(\omega'_{K^\gamma})},
\]
and so by \eqref{eqn:wDef},
\begin{align}\label{eqn:vhStable}
\|\bv\|_{H^1(K)}\lesssim \|\bw\|_{H^1(\omega'_{K_\gamma})}\ \forall K\in \calT_h\quad \Longrightarrow\quad \|\bv\|_{H^1_h(\Gamma_h)}\lesssim \|\bw\|_{H^1(\gamma)}\lesssim \|q\|_{L_2(\Gamma_h)}.
\end{align}

Next, we recall from Proposition \ref{prop:Hdiv} that $\bv\in \bH({\rm div}_{\Gamma_h};\Gamma_h)$
and therefore \eqref{eqn:HdivProp} is satisfied.
Thus, by integration by parts, the identity $\int_K \bv = \int_K \ipt{\bw}$,
and applying 
\eqref{eqn:PiolaInv1} yields
\begin{align*}
\int_{\Gamma_h} ({\rm div}_{\Gamma_h} \bv) q
&= -\int_{\Gamma_h} \bv\cdot \nab_{\Gamma_h} q 
= -\int_{\Gamma_h} \ipt\bw \cdot \nab_{\Gamma_h} q 
= \int_{\Gamma_h} ({\rm div}_{\Gamma_h} \ipt \bw)  q 
 = \int_\gamma  ({\rm div}_\gamma \bw) q^\ell. 
\end{align*}
We then use \eqref{eqn:wDef}, \eqref{eqn:muBound}, and \eqref{eqn:extEquivalence}
to obtain
\begin{align*}
\int_{\Gamma_h} ({\rm div}_{\Gamma_h} \bv) q
\gtrsim \|q\|_{L_2(\Gamma_h)}^2.
\end{align*}
This identity combined with \eqref{eqn:vhStable} completes the proof.
\end{proof}



\subsection{$\textit{\textbf{H}}^{\hspace{1pt} 1}_T$-conforming approximations to discrete functions}
 While the finite element space
$\bV_h$ is merely $\bH({\rm div}_{\Gamma_h};\Gamma_h)$-conforming (cf.~Proposition \ref{prop:Hdiv}),
the following lemma shows that functions
in this space are ``close'' to an $\bH^1$-conforming relative.
 


\begin{lemma}\label{lem:H1_approx}
Given $\bv \in \bV_h$, denote by $\pt{\bv} = \calP_{\bp} \bv$ its Piola transform via the closest point projection to $\gamma$.  Then there exists $\pt \bv_c \in \bH_T^1(\gamma)$ such that 
\begin{align}
\label{H1_approx}
\|\pt{\bv} - \pt{\bv}_c \|_{L_2(K^\gamma)} + h |\pt{\bv} - \pt \bv_c |_{H_h^1(K^\gamma)} \lesssim h^2 \|\pt{\bv} \|_{L_2(K^\gamma)}\quad \forall K^\gamma \in \calT_h^\gamma.
\end{align}
\end{lemma}
\begin{proof}
On an element $K\in \calT_h$, we first write $\bv=\bv_{L} + \balpha b_K$, where $\bv_{L}$ is componentwise 
affine on $K$, and ${\balpha}\in \bbR^3$ is tangent to $K$ (cf.~\eqref{eqn:MiniDecomp}).  Likewise, $\pt \bv = \pt \bv_{L} + \pt{\balpha} b_K^\ell$ is
the Piola transform of $\bv$ to $\gamma$.   We next let $\bw$ be the unique
continuous piecewise linear polynomial with respect to $\calT_h$ satisfying $\bw(a)=\bv_{K_a}(a)$ for each vertex $a\in \calV_h$.
We then set
\begin{align*}
\pt \bv_c=\frac{\bPi-d\bH}{(1-d \kappa_1)(1-d\kappa_2)} \bw^\ell + \pt{\balpha} b_K^\ell = \mu_h^{-1} \bnu\cdot \bnu_h (\bPi-d\bH) \bw^\ell + \pt{\balpha} b_K^\ell.
\end{align*}
Note that $\pt \bv_{c} \in \bH_T^1(\gamma)$, and 
\begin{align*}
\pt \bv - \pt \bv_c = \pt \bv_{L} - \mu_h^{-1} \bnu\cdot \bnu_h (\bPi-d\bH) \bw^\ell.
\end{align*}
Fixing $K \in \calT_h$, by norm equivalence (cf.~\eqref{eqn:PiolaNormEquiv}) we  prove \eqref{H1_approx} by establishing that 
\begin{align}
\label{toprove}
\|\bv_{L} - \iptW {\mu_h^{-1} \bnu \cdot \bnu_h (\bPi-d\bH) \bw^\ell} \|_{L_2(K)} + h |\bv_{L} - \iptW {\mu_h^{-1} \bnu \cdot \bnu_h (\bPi-d\bH) \bw^\ell}  |_{H_h^1(K)} \lesssim h^2 \|\bv \|_{L_2(K)},
\end{align}
{where $\iptW {\mu_h^{-1} \bnu \cdot \bnu_h (\bPi-d\bH) \bw^\ell} = \calP_{\bp|_K^{-1}}\mu_h^{-1} \bnu \cdot \bnu_h (\bPi-d\bH) \bw^\ell$.} 
We show \eqref{toprove} in three steps.

(i) Employing \eqref{eqn:invPiola}, \eqref{eqn:muForm}, and $\bPi_h (\bnu \cdot \bnu_h - \bnu \otimes \bnu_h) \bPi = \bnu \cdot \bnu_h - \bnu \otimes \bnu_h$, we have that
\begin{equation*}
\iptW {\mu_h^{-1} \bnu \cdot \bnu_h (\bPi-d\bH) \bw^\ell} = \bPi_K \left((\bnu \cdot \bnu_K) {{\bf I}}-\bnu \otimes \bnu_K\right) \bw.
\end{equation*}
Here we interchangeably write $\bnu_h=\bnu_K=\bnu_h|_K$ and $\bPi_h =\bPi_K$ in order to better distinguish dependence on the element $K$.  
Using $\bv_L \cdot \bnu_K=0$ and $\bPi_K \bv_L=\bv_L$, we then have
\begin{align}\label{error_rep}
\begin{aligned}
\bv_L- & \iptW {\mu_h^{-1} \bnu \cdot \bnu_h (\bPi-d\bH) \bw^\ell} =\bPi_K  [ \bv_L-\left((\bnu \cdot \bnu_K){{\bf I}}-\bnu \otimes \bnu_K\right) \bw]
\\ & = [(1-\bnu \cdot \bnu_K) \bv_L ]+ [\bnu \cdot \bnu_K \bPi_K (\bv_L-\bw) ]+  [ \bnu_K \cdot (\bw- \bv_L) \bPi_K (\bnu-\bnu_K)]
\\ & =: I + II + III.
\end{aligned}
\end{align}

(ii)  We next bound the terms $I$, $II$, and $III$ in $L_2$.  Using $|1-\bnu \cdot \bnu_K| = \frac{1}{2} |\bnu-\bnu_K|^2 \lesssim h^2$ yields
\begin{align}\label{eqn:IBound}
\|I\|_{L_2(K)}  \lesssim h^2 \|\bv_L\|_{L_2(K)}.
\end{align}
Next we use \eqref{eqn:calMDef} and recall that $\bv_{K_a}(a) \cdot \bnu_{K_a}=0$ to compute that for each vertex $a \in K$
\begin{align}
\label{computation1}
\begin{aligned}
|\bPi_K (\bv_L-\bw)(a)| & = |\bPi_K (\calM_a^K-{\bf I}) \bv_{K_a}(a) |
\\ & = |\bPi_K  [\left((\bnu_{K_a}\cdot \bnu_K) {{\bf I}}-\bnu_{K_a} \otimes \bnu_K\right)-{\bf I}]\bv_{K_a}(a)|
\\& = |(\bnu_{K_a} \cdot \bnu_{K}-1) \bPi_K \bv_{K_a}(a)-  (\bnu_{K}-\bnu_{K_a}) \cdot \bv_{K_a}(a) \bPi_K (\bnu_{K_a}-\bnu_K)|
\\ & \lesssim h^2 |\bv_{K_a}(a)| \lesssim {h^2} |\bv_K(a)|.
\end{aligned}
\end{align}
In the last step we have employed \eqref{eqn:Piola} and \eqref{eqn:invPiola} to obtain $|\bv_{K_a}(a)| = |\frac{1}{\bnu_K \cdot \bnu_{K_a}} \bPi_{K_a} \bv_K(a)| \lesssim |\bv_K(a)|$.  We then use the fact that $\bPi_K(\bv_L-\bw)$ and $\bv_L$ are affine, along with inverse inequalities, to obtain
\begin{align}
\label{computation2}
\begin{aligned}
\|II\|_{L_2(K)} &  \lesssim  \|\bPi_K(\bv_L-\bw)\|_{L_2(K)} \lesssim h \max_{a\in {\calV_K}} |\bPi_K (\bv_L-\bw)(a)| 
\\ & \lesssim h^3 \|\bv_L\|_{L_\infty(K)} \lesssim h^2\|\bv_L\|_{L_2(K)}.
\end{aligned}
\end{align}
In order to bound $III$, we first proceed similarly {as} \eqref{computation1} to obtain
\begin{align*}
|(\bv_L-\bw)(a)|= |(\bnu_{K_a} \cdot \bnu_{K}-1) \bv_{K_a}(a) - ( \bnu_K-\bnu_{K_a}) \cdot \bv_{K_a}(a) \bnu_{K_a}| \lesssim h|\bv_K(a)|.
\end{align*}
Using $|\bPi_K (\bnu-\bnu_K)| \lesssim h$, we thus have similar to above that
\begin{align}
\label{computation3}
\|III\|_{L_2(K)} \lesssim h \|\bv_L-\bw\|_{L_2(K)} \lesssim h^2 \|\bv_L\|_{L_2(K)}.  
\end{align}
Recalling that $\bv=\bv_L+ \balpha b_K$ and $b_K(a)=0$ at vertices $a$, we again use inverse inequalities to obtain 
\begin{align}
\label{computation4}
\|\bv_L\|_{L_2(K)} \lesssim h \|\bv_L\|_{L_\infty(K)}= h \max_{a \in \calV_K} |\bv_L(a)| \le h \|\bv\|_{L_\infty(K)} {\lesssim} \|\bv\|_{L_2(K)}.
\end{align}
Collecting the  inequalities \eqref{eqn:IBound}, \eqref{computation2}--\eqref{computation4}
yields $\|\pt \bv- \pt{\bv}^c\|_{L_2(K^\gamma)} \lesssim h^2 \|\pt \bv\|_{L_2(K^\gamma)}.$

(iii) In order to bound $\|{\nabla_{\Gamma_h}}( \bv_L-  \iptW {\mu_h^{-1} \bnu \cdot \bnu_h (\bPi-d\bH) \bw^\ell})\|_{L_2(K)}$, we recall \eqref{error_rep} and consider first $\nabla_{\Gamma_h} I$.  First note that $|\nabla (1-\bnu \cdot \bnu_K)| = |\bH \bnu_K| = |\bH (\bnu_K-\bnu)| \lesssim h$.  Thus using an inverse inequality and $|1-\bnu\cdot \bnu_K| \lesssim h^2$, we obtain
\begin{align*}
\|\nabla_{\Gamma_h} I\|_{L_2(K)} \lesssim \|1-\bnu \cdot \bnu_K\|_{L_\infty(K)} \|\nabla_{\Gamma_h} \bv_L \|_{L_2(K)} 
+ \|\nabla(1-\bnu \cdot \bnu_K)\|_{L_\infty(K)} \|\bv_L\|_{L_2(K)} \lesssim h \|\bv_L\|_{L_2(K)}.
\end{align*}
Employing inverse inequalities, $|\nabla (\bnu \cdot \bnu_K)| \lesssim h$, and \eqref{computation2} also yields
\begin{align*}
\begin{aligned}
\|\nabla_{\Gamma_h} II \|_{L_2(K)} & \le \|\nabla ( \bnu \cdot \bnu_K)\|_{L_\infty(K)} \|\bPi_K (\bv_L-\bw)\|_{L_2(K)} + \|\bnu \cdot \bnu_K\|_{L_\infty(K)} \|\nabla_{\Gamma_h} [\bPi_K(\bv_L-\bw)]\|_{L_2(K)} 
\\ & \lesssim h\|\bPi_K (\bv_L-\bw)\|_{L_2(K)} + h^{-1} \| \bPi_K (\bv_L-\bw)\|_{L_2(K)} 
\\ & \lesssim h \|\bv_L\|_{L_2(K)}.
\end{aligned}
\end{align*}
We finally compute using inverse inequalities and \eqref{computation3} that
\begin{align*}
\begin{aligned}
\|\nabla_{\Gamma_h} III\|_{L_2(K)} & \lesssim \|\nabla [\bPi_K (\bnu-\bnu_K)]\|_{L_\infty(K)} \|\bv_L-\bw\|_{L_2(K)}
\\ & ~~~~~~~ + \|\bPi_K(\bnu-\bnu_K)\|_{L_\infty(K)} \|\nabla_{\Gamma_h}(\bv_L-\bw)\|_{L_2(K)}
\\ & \lesssim (1+h h^{-1}) \|\bv_L-\bw\|_{L_2(K)} \lesssim h \|\bv_L\|_{L_2(K)}.
\end{aligned}
\end{align*}
Collecting the above {inequalities}
and employing \eqref{computation4} yields
\[
\|\nabla_{\Gamma_h} (I+II+III)\|_{L_2(K)} \lesssim h \|\bv_L\|_{L_2(K)} \lesssim h \|\bv\|_{L_2(K)},
\]
which completes the proof.  
\end{proof}

\subsection{Discrete Korn-type inequalities}

From Lemma \ref{lem:H1_approx}, we immediately obtain
a discrete Korn-type inequality on the exact surface $\gamma$.
\begin{lemma}\label{eqn:Korn1}  Given $\bv \in \bV_h$, there holds
\begin{align}
\label{discrete_korn}
\|\pt \bv \|_{H_h^1(\gamma)} \lesssim \|\pt \bv \|_{L_2(\gamma)} + \|{\rm Def}_{\gamma,h} \pt \bv \|_{L_2(\gamma)},
\end{align}
where $\pt \bv = \calP_{\bp} \bv$.
\end{lemma}
\begin{proof}
Given $\bv \in \bV_h$, let $\pt\bv_c \in \bH_T^1(\gamma)$ satisfy \eqref{H1_approx}.  
A continuous Korn inequality holds for $\pt \bv_c$, so using \eqref{H1_approx} we have
\begin{equation}\label{eqn:Triangle}
\begin{split}
\|\pt \bv \|_{H_h^1(\gamma)} & \lesssim \|\pt \bv_c \|_{H^1(\gamma)} + \|\pt\bv_c - \pt \bv\|_{H_h^1(\gamma)} 
\\ & \lesssim \|\pt \bv_c \|_{L_2(\gamma)} + \|{\rm Def}_\gamma \pt \bv_c \|_{L_2(\gamma)} + \|\pt \bv_c - \pt \bv\|_{H_h^1(\gamma)} 
\\ & \lesssim \|\pt\bv \|_{L_2(\gamma)} + \|{\rm Def}_{\gamma, h} \pt\bv \|_{L_2(\gamma)} + \|\pt \bv_c - \pt\bv\|_{H_h^1(\gamma)} 
\\ & \lesssim (1+h) \|\pt\bv \|_{L_2(\gamma)} + \|{\rm Def}_{\gamma, h}\pt \bv \|_{L_2(\gamma)}.
\end{split}
\end{equation}
\end{proof}

From this result, we obtain a discrete Korn inequality
for $\bV_h$ on $\Gamma_h$.
\begin{lemma}
There holds
\begin{align}\label{eqn:KornOnGammah}
\|\bv\|_{H^1_h(\Gamma_h)}\lesssim \|\bv\|_{L_2(\Gamma_h)}+ \|{\rm Def}_{\Gamma_h} \bv\|_{L_2(\Gamma_h)}\qquad \forall \bv\in \bV_h.
\end{align}
\end{lemma}
\begin{proof}
We apply Lemma \ref{eqn:Korn1}, \eqref{eqn:DefRelation}, and an inverse estimate:
\begin{align*}
\|\bv\|_{H^1_h(\Gamma_h)}
&\lesssim \|\pt \bv\|_{H^1_h(\gamma)} \\ 
&\lesssim \| \pt \bv\|_{L_2(\gamma)}+\|{\rm Def}_{\gamma,h} \pt \bv\|_{L_2(\gamma)}\\
&\lesssim \| \bv\|_{L_2(\Gamma_h)}+\|{{\rm Def}_{\Gamma_h}}  \bv\|_{L_2(\Gamma_h)}+ h \|\bv\|_{H^1_h(\Gamma_h)}\\
&\lesssim \| \bv\|_{L_2(\Gamma_h)}+\|{{\rm Def}_{\Gamma_h}}  \bv\|_{L_2(\Gamma_h)}.
\end{align*}\hfill
\end{proof}

\section{Finite element method and convergence analysis}\label{sec-FEM}

For piecewise smooth $\bw,\bv$ with $\bv\in \bH({\rm div}_\gamma;\gamma)$
and $q\in L_2(\gamma)$,
we define the bilinear forms
\begin{align*}
a_\gamma(\bw,\bv) 
&= \int_\gamma {\rm Def}_{\gamma,h} \bw: {\rm Def}_{\gamma,h} \bv +\int_\gamma \bw\cdot \bv,\\
b_\gamma(\bv,q) 
& = -\int_\gamma ({\rm div}_\gamma \bv)q.
\end{align*}
The variational formulation for the Stokes problem \eqref{eqn:Stokes}
seeks $(\bu,p) \in \bH_T^1(\gamma)\times \mathring{L}_2(\gamma)$ satisfying
\begin{equation}
\label{eqn:StokesVar}
\begin{aligned}
a_\gamma(\bu,\bv) + b_\gamma(\bv,p) & = \int_\gamma {\bm f}\cdot \bv\qquad &&\forall \bv\in \bH_T^1(\gamma),\\
b_\gamma(\bu,q) & = 0\qquad &&\forall q\in \mathring{L}_2(\gamma).
\end{aligned}
\end{equation}
\begin{remark}\label{rem-Killing}
In order to ensure the well-posedness of \eqref{eqn:StokesVar} and avoid technical complications associated with Killing fields, 
we include the zeroth-order mass term in the momentum equations, as mentioned earlier in the introduction. 
A method for incorporating Killing fields 
into surface finite element methods for the  Stokes problem is presented in \cite{SurfaceStokes1},
and the main ideas presented there are applicable to the proposed discretization below.
\end{remark}

We define the analogous bilinear forms 
with respect to the discrete surface $\Gamma_h$:
\begin{align*}
a_{\Gamma_h}(\bw,\bv) 
&= \int_{\Gamma_h} {\rm Def}_{\Gamma_h} \bw: {\rm Def}_{\Gamma_h} \bv +\int_{{\Gamma_h}} \bw\cdot \bv,\\
b_{\Gamma_h}(\bv,q) 
& = -\int_{\Gamma_h} ({\rm div}_{\Gamma_h} \bv)q,
\end{align*}
where the differential operator ${\rm Def}_{\Gamma_h}$ is understood to act piecewise
with respect to $\calT_h$.  Then 
the finite element method seeks $(\bu_h,p_h)\in \bV_h\times Q_h$ such that
\begin{equation}\label{eqn:FEM}
\begin{aligned}
a_{\Gamma_h}(\bu_h,\bv) + b_{\Gamma_h}(\bv,p_h) & = \int_{{\Gamma_h}} {\bm f}_h\cdot \bv\qquad &&\forall \bv\in \bV_h,\\
b_{\Gamma_h}(\bu_h,q) & = 0\qquad &&\forall q\in Q_h,
\end{aligned}
\end{equation}
where ${\bm f}_h$ is some approximation of ${\bm f}$ that is defined on $\Gamma_h$.

By the inf-sup condition \eqref{eqn:dInfSup}, the discrete Korn-like inequality \eqref{eqn:KornOnGammah},
and standard theory of saddle-point problems,
there exists a unique solution \eqref{eqn:FEM}.
To derive error estimates, we restrict
 \eqref{eqn:FEM} to the discretely divergence--free subspace
$\bX_h := \{\bv\in \bV_h:\ \int_{\Gamma_h} ({\rm div}_{\Gamma_h} \bv)q = 0\ \forall q\in Q_h\}$.
Then $\bu_h\in \bX_h$ is uniquely determined by the problem
\begin{equation}\label{eqn:PoissonFEM}
a_{\Gamma_h}(\bu_h,\bv) = 
 \int_{\Gamma_h} {\bm f}_h \cdot \bv\quad \forall \bv\in \bX_h.
 \end{equation}
 Now set $\pt{\bu}_h = \calP_{\bp} \bu_h$,
$\pt{\bv} = \calP_{\bp} \bv$, and note that
\begin{align*}
\int_{\Gamma_h} {\bm f}_h\cdot \bv = \int_{\Gamma_h} {\bm f}_h \cdot \calP_{\bp^{-1}} \pt{\bv} = \int_{\gamma} {\bm F}_h \cdot \pt{\bv},
\end{align*}
where ${\bm F}_h =({\bf M}^\intercal {\bm f}_h)^\ell $, and 
  \begin{equation}\label{eqn:MDef}
  {\bf M} = \Big[{\bf I}-\frac{\bnu\otimes \bnu_h}{\bnu\cdot \bnu_h}\Big][{\bf I}-d {\bf H}]^{-1}
  \end{equation}
  is the matrix arising in the definition of $\calP_{\bp^{-1}}$.
  Therefore \eqref{eqn:PoissonFEM} is equivalent to the statement
\begin{equation}\label{eqn:FEMXhPt}
a_\gamma(\pt{\bu}_h,\pt{\bv}) = 
 \int_{\gamma} {{\bm F}_h\cdot \pt{\bv}} +G_h(\bu_h,\bv)\quad \forall {\bv}\in {\bX}_h,
 \end{equation}
 where the bilinear form $G_h: {\bH^1_h(\Gamma_h)\times \bH^1_h(\Gamma_h)}
 \to \mathbb{R}$ 
 given by
 \[
 G_h(\bw,\bv) =a_\gamma(\pt{\bw},\pt{\bv})-a_{\Gamma_h}(\bw,\bv)
 \]
  encodes geometric error.
  \begin{lemma}\label{lem:GhBound}
There holds
\begin{equation}\label{eqn:GhBound}
|G_h(\bw,\bv)|\lesssim h \|\pt \bw\|_{H^1_h(\gamma)} \|\pt{\bv}\|_{H^1_h(\gamma)}
\end{equation}
for all tangential $\bw,\bv\in \bH^1_h(\Gamma_h)$.
\end{lemma}
\begin{proof}
We write
\begin{align*} 
&\int_{\gamma} {\rm Def}_{\gamma,h} \pt{\bw}:{\rm Def}_{\gamma,h} \pt{\bv} - \int_{\Gamma_h} {\rm Def}_{\Gamma_h} \bw:{\rm Def}_{\Gamma_h} \bv\\
&\qquad = \int_{\gamma} {\rm Def}_{\gamma,h} \pt{\bw}:{\rm Def}_{\gamma,h} \pt{\bv} - \int_{\gamma} (\mu_h^{-1} {\rm Def}_{\Gamma_h} \bw)\circ \bp^{-1} :({\rm Def}_{\Gamma_h} \bv)\circ \bp^{-1}\\
&\qquad = \int_{\gamma} \left({\rm Def}_{\gamma,h} \pt{\bw}-(\mu_h^{-1} {\rm Def}_{\Gamma_h} \bw)\circ \bp^{-1}\right):{\rm Def}_{\gamma,h} \pt{\bv}\\
&\qquad\qquad - \int_{\gamma} (\mu_h^{-1} {\rm Def}_{\Gamma_h} \bw)\circ \bp^{-1} :\left(({\rm Def}_{\Gamma_h} \bv)\circ \bp^{-1}-{\rm Def}_{\gamma,h} \pt{\bv}\right).
\end{align*}
Applying \eqref{eqn:DefDef}, \eqref{eqn:muBound}, and Lemma \ref{lem:PiolaNormEquiv}, we obtain
\begin{align}\label{eqn:GBound1}
\left|\int_{\gamma} {\rm Def}_{\gamma,h} \pt{\bw}:{\rm Def}_{\gamma,h} \pt{\bv} - \int_{\Gamma_h} {\rm Def}_{\Gamma_h} \bw:{\rm Def}_{\Gamma_h} \bv\right|
&\lesssim h \|\pt{\bw}\|_{H^1_h(\gamma)} \|\pt{\bv}\|_{H^1_h(\gamma)}.
\end{align}

Next, we use the formula of the Piola transform
involving ${\bf M}$ to obtain
\begin{align*}
\int_\gamma \pt{\bw}\cdot \pt{\bv} - \int_{\Gamma_h} \bw\cdot \bv
& = \int_\gamma \pt{\bw}\cdot \pt{\bv}  -\int_\gamma (\mu_h\circ \bp^{-1})  \pt{\bw}^{\intercal} {\bf M}^\intercal {\bf M} \pt{\bv}\\
%
%
& = \int_\gamma \pt{\bw}^{\intercal}[\bPi-(\mu_h \circ \bp^{-1}) {\bf M}^{\intercal} {\bf M}] \pt{\bv}. 
\end{align*}
A short  
computation using \eqref{eqn:muBound} yields $|\bPi-(\mu_h \circ \bp^{-1}) {\bf M}^{\intercal} {\bf M}| \lesssim |(\bnu-\frac{\bnu_h}{\bnu \cdot \bnu_h}) \otimes (\bnu-\frac{\bnu_h}{\bnu \cdot \bnu_h})| + h^2 \lesssim h^2$.  Thus 
%
\begin{equation}\label{eqn:GBound2}
\left|\int_\gamma \pt{\bw}\cdot \pt{\bv} - \int_{\Gamma_h} \bw\cdot \bv\right|\lesssim h^2 \|\pt{\bw}\|_{L_2(\gamma)}\|\pt{\bv}\|_{L_2(\gamma)}.
\end{equation}
The result  \eqref{eqn:GhBound} follows from \eqref{eqn:GBound1}--\eqref{eqn:GBound2}.
\end{proof}

The next lemma states the approximation properties
of the discretely divergence--free subspace $\bX_h$.
The result essentially follows from the inf-sup condition \eqref{eqn:dInfSup}
and the arguments
in \cite[Theorem 12.5.17]{BrennerScott}. For completeness
we provide the proof.
\begin{lemma}\label{lem:XhApprox}
Let $\bu\in \bH^1(\gamma)$ satisfy $\Div_{\gamma} \bu=0$.
Then there holds
\[
\inf_{\bv\in \bX_h} \|\bu-\pt{\bv}\|_{H^1_h(\gamma)}\lesssim \inf_{\bv\in \bV_h} \|\bu-\pt{\bv}\|_{H^1_h(\gamma)}.
\]
\end{lemma}
\begin{proof}
Fix $\bv\in \bV_h$.  The inf-sup condition \eqref{eqn:dInfSup} implies  there exists $\bw\in \bV_h$ such that
$b_{\Gamma_h}(\bw,q) = b_\gamma(\bu-\pt{\bv},q^\ell)$ for all $q\in Q_h$, and $\|\bw\|_{H^1_h(\Gamma_h)}\lesssim \|\bu-\pt{\bv}\|_{H^1_h(\gamma)}$.
Then $\bw+\bv\in \bX_h$ and $\|\bu-(\pt{\bw}+\pt{\bv})\|_{H^1_h(\gamma)}\le \|\bu-\pt{\bv}\|_{H^1_h(\gamma)}+\|\pt{\bw}\|_{H^1_h(\gamma)}\lesssim \|\bu-\pt{\bv}\|_{H^1_h(\gamma)}$.
This implies the desired result.
\end{proof}

\begin{theorem}
Let $(\bu_h,p_h)\in \bV_h\times Q_h$
satisfy the finite element method \eqref{eqn:FEM}.
Let $\pt{\bu}_h = \calP_{\bp} \bu_h$ denote the Piola
transform of $\bu_h$ with respect to the closest point projection
$\bp$, and let $p^\ell_h = p_h\circ \bp^{-1}$. Then there holds
\begin{subequations}
\begin{align}
\label{eqn:uhErrorAb}
\|\bu-\pt{\bu}_h\|_{H^1_h(\gamma)}
&\lesssim \inf_{(\bv,q)\in \bV_h\times Q_h} \big(\|\bu-\pt{\bv}\|_{H^1_h(\gamma)}+ \|p-q^\ell\|_{L_2(\gamma)}\big)+ h^2 \|{\bm f}\|_{L_2(\gamma)} +\|{\bm f}-{\bm F}_h\|_{L_2(\gamma)}\\
 &\nonumber \quad+h\big(\|p\|_{L_2(\gamma)}+ \|\bu\|_{H^1(\gamma)}+\|{\bm f}_h\|_{L_2(\Gamma_h)}\big),\\
 \label{eqn:phErrorAb}
 \|p-p_h^\ell\|_{L_2(\gamma)} 
 &\lesssim \inf_{q\in Q_h} \|p-q^\ell\|_{L_2(\gamma)}
+\|\bu-\pt{\bu}_h\|_{H^1_h(\gamma)}+ 
h^2 \|{\bm f}\|_{L_2(\gamma)} +\|{\bm f}-{\bm F}_h\|_{L_2(\gamma)}\\
 &\nonumber \quad+h\big(\|p\|_{L_2(\gamma)}+ \|\bu\|_{H^1(\gamma)}+\|{\bm f}_h\|_{L_2(\Gamma_h)}\big).
 \end{align}
 \end{subequations}
Therefore, by Lemma \ref{lem:Interp}, if $(\bu,p)\in \bH^2(\gamma)\times H^1(\gamma)$, there holds
\begin{align}
\label{convergence}
\|\bu-\pt{\bu}_h\|_{H^1_h(\gamma)}+\|p-p_h^\ell\|_{L_2(\gamma)}
&\lesssim h(\|\bu\|_{H^2(\gamma)}+\|p\|_{H^1(\gamma)}+\|{\bm f}_h\|_{L_2(\Gamma_h)}) +\|{\bm f}-{\bm F}_h\|_{L_2(\gamma)}.
\end{align}
\end{theorem}
\begin{proof}
For $\bv\in \bX_h$, we denote by $\pt{\bv}_c\in \bH_T^1(\gamma)$ the conforming
relative of $\pt{\bv} = \calP_{\bp} \bv$ satisfying \eqref{H1_approx}.
Using \eqref{eqn:StokesVar}, \eqref{eqn:FEMXhPt} and \eqref{eqn:PiolaInv1}, we write 
\begin{align*}
a_\gamma(\bu-\pt{\bu}_h,\pt{\bv})
& = a_\gamma(\bu,\pt{\bv}_c) + a_\gamma(\bu,\pt{\bv}-\pt{\bv}_c) - \int_\gamma {\bm F}_h \cdot \pt{\bv} - G_h(\bu_h,\bv)\\
& = \int_\gamma {\bm f}\cdot \pt{\bv}_c  - \int_\gamma {\bm F}_h \cdot \pt{\bv}-b_\gamma(\pt{\bv}_c,p)+ a_\gamma(\bu,\pt{\bv}-\pt{\bv}_c) - G_h(\bu_h,\bv)\\
&  = \int_{\gamma} {\bm f}\cdot (\pt{\bv}_c-\pt{\bv}) -\int_\gamma ({\bm F}_h-{\bm f}) \cdot \pt{\bv}
  -b_\gamma(\pt{\bv},p-q^\ell)  -b_\gamma(\pt{\bv}_c-\pt{\bv},p)\\
  &\qquad+ a_\gamma(\bu,\pt{\bv}-\pt{\bv}_c) - G_h(\bu_h,\bv)\qquad \forall q\in Q_h.
\end{align*}
Applying continuity estimates of the bilinear forms,
\eqref{H1_approx}, and \eqref{eqn:GhBound} yield
\begin{align*}
&a_\gamma(\bu-\pt{\bu}_h,\pt{\bv})\\
&~~\lesssim  
\big(h^2 \|{\bm f}\|_{L_2(\gamma)} +\|{\bm f}-{\bm F}_h\|_{L_2(\gamma)} +\|p-q^\ell\|_{L_2(\gamma)}
 +h\|p\|_{L_2(\gamma)}+h \|\bu\|_{H^1(\gamma)}+ h\|\pt{\bu}_h\|_{H^1_h(\gamma)}\big)
\|\pt \bv\|_{H^1_h(\gamma)}.
\end{align*}
The estimate $\|\pt{\bu}_h\|_{H^1_h(\gamma)}\lesssim \|\bu_h\|_{H^1_h(\Gamma_h)}\lesssim \|{\bm f}_h\|_{L_2(\Gamma_h)}$, and standard arguments then yield
\begin{align*}
\|\bu-\pt{\bu}_h\|_{H^1_h(\gamma)}
&\lesssim \inf_{(\bv,q)\in \bX_h\times Q_h} \big(\|\bu-\pt{\bv}\|_{H^1_h(\gamma)}+ \|p-q^\ell\|_{L_2(\gamma)}\big)
+ h^2 \|{\bm f}\|_{L_2(\gamma)} +\|{\bm f}-{\bm F}_h\|_{L_2(\gamma)}\\
 &\qquad+h(\|p\|_{L_2(\gamma)}+ \|\bu\|_{H^1(\gamma)} +\|{\bm f}_h\|_{L_2(\Gamma_h)}).
\end{align*}
The estimate \eqref{eqn:uhErrorAb} then 
follows by applying Lemma \ref{lem:XhApprox}.

For the pressure error,
we similarly apply \eqref{eqn:PiolaInv1},  \eqref{eqn:FEMXhPt}, \eqref{H1_approx}, and \eqref{eqn:FEMXhPt}--\eqref{eqn:GhBound}
to obtain for all $\bv\in \bV_h$ and $q\in Q_h$,
\begin{align*}
b_{\Gamma_h}(\bv,p_h-q)
& = \int_{\Gamma_h} {\bm f}_h \cdot \bv - a_{\Gamma_h}(\bu_h,\bv)-b_\gamma(\pt{\bv},q^\ell)\\
& = \int_{\gamma} {\bm F}_h \cdot \pt{\bv} - a_\gamma(\pt{\bu}_h,\pt{\bv}) -b_\gamma(\pt{\bv},q^\ell) +G_h(\bu_h,\bv)\\
%
%
& = \int_\gamma {\bm f}\cdot (\pt{\bv}-\pt{\bv}_c)+\int_{\gamma} ({\bm F}_h -{\bm f})\cdot \pt{\bv} +a_\gamma(\bu - \pt{\bu}_h,\pt{\bv}) +b_\gamma(\pt{\bv},p-q^\ell) +G_h(\bu_h,\bv)\\
&\qquad -a_\gamma(\pt{\bu},\pt{\bv}-\pt{\bv}_c)-b_\gamma(\pt{\bv}-\pt{\bv}_c,p)\\
&\lesssim \Big(h^2 \|{\bm f}\|_{L_2(\gamma)} +\|{\bm f}-{\bm F}_h\|_{L_2(\gamma)}
+\|\bu-\pt{\bu}_h\|_{H^1_h(\gamma)}+ \|p-q^\ell\|_{L_2(\gamma)}\\
&\qquad+ h\big(\|\pt\bu_h\|_{H^1_h(\gamma)}+\|\bu\|_{H^1(\gamma)}+\|p\|_{L_2(\gamma)}\big)\Big) \|\pt{\bv}\|_{H^1_h(\gamma)}.
\end{align*}
%
We conclude from the inf-sup condition \eqref{eqn:dInfSup}
and the estimates $\|\pt{\bu}_h\|_{H^1_h(\gamma)}\lesssim \|{\bm f}_h\|_{L_2(\Gamma_h)}$,
$\|\pt{\bv}\|_{H^1_h(\gamma)}\lesssim \|\bv\|_{H^1_h(\Gamma_h)}$
that
\begin{align*}
\|p-p_h^\ell \|_{L_2(\gamma)}
&\le \|p-q^\ell\|_{L_2(\gamma)}+\|p_h^\ell-q^\ell\|_{L_2(\gamma)}\\
&\lesssim \|p-q^\ell\|_{L_2(\gamma)}+\|p_h-q\|_{L_2(\Gamma_h)}\\
&\lesssim  h^2 \|{\bm f}\|_{L_2(\gamma)} +\|{\bm f}-{\bm F}_h\|_{L_2(\gamma)}
+\|\bu-\pt{\bu}_h\|_{H^1_h(\gamma)}+ \|p-q^\ell\|_{L_2(\gamma)}\\
&\qquad+ h\big(\|\bu\|_{H^1(\gamma)}+\|p\|_{L_2(\gamma)}+\|{\bm f}_h\|_{L_2(\gamma)}\big). 
\end{align*}
By taking the infimum over $q\in Q_h$ we obtain \eqref{eqn:phErrorAb}.
\end{proof}

\begin{remark}In order to obtain a final $O(h)$ energy error bound from \eqref{convergence} we must choose ${\bm f}_h$ so that $\|{\bm f}-{\bm F}_h\|_{L_2(\gamma)} \lesssim h$.  A short calculation shows that ${\bm f}_h=\calP_{\bp^{-1}} {\bm f}  $ yields $\|{\bm f}-{\bf F}_h\|_{L_2(\gamma)} \lesssim h^2$; a variety of other choices also yield optimal convergence.
\end{remark}

\begin{remark}  
Analysis of $L_2$ errors in the velocity is the subject of ongoing work.  
Numerical experiments presented below indicate that $\|\ipt \bu - \bu_h\|_{L_2(\Gamma_h)} \lesssim h^2$, as expected.  
However, the conforming approximation error estimate given in Lemma \ref{lem:H1_approx} {seems} insufficient 
to obtain an $O(h^2)$ convergence rate in $L_2$.  In addition, the $O(h)$ geometric error estimate in Lemma \ref{lem:DefRelation} is sufficient to establish optimal $O(h)$ convergence in the energy norm, but not an optimal $O(h^2)$ $L_2$ convergence rate.  Obtaining $O(h^2)$ geometric error estimates sufficient to achieve optimal $L_2$ convergence is likely possible using techniques introduced in \cite{HansboLarsonLarsson20} but is significantly more technical than the energy case analyzed here.  
\end{remark}

\section{Numerical Experiments}\label{sec-numerics}

In this section we briefly comment on 
{the implementation of the finite element method \eqref{eqn:FEM}}
and then present numerical experiments demonstrating optimal convergence rates in the energy and $L_2$ norms for both MINI and lowest-order Taylor-Hood elements.

\subsection{Implementation notes}

The main additional complication in the implementation of the surface MINI method as compared to the Euclidean case arises in choosing two individual degrees of freedom at each vertex and interpreting them on each incident element.   In the Euclidean case it is natural to choose individual degrees of freedom to align with the canonical Euclidean basis vectors.  Thus natural global basis functions corresponding to a vertex $a$ are $(\varphi_a,0)^{\top} $ and $(0, \varphi_a)^{\top}$, with $\varphi_a$ the usual affine hat function corresponding to $a$.  In the surface case there is no such natural choice, so at each vertex $a \in \calV_h$ one must first fix the master element $K_a$ along with two arbitrary but mutually orthogonal unit vectors $\bv_{1,a,K_a}$ and $\bv_{2,a,K_a}$ tangent to $K_a$.   The two global degrees of freedom corresponding to $a$ are then $\bu_h|_{K_a}(a) \cdot \bv_{1,a,K_a}$ and $\bu_h|_{K_a} (a) \cdot \bv_{2,a,K_a}$.  

Once $K_a$, $\bv_{1,a,K_a}$, and $\bv_{2,a,K_a}$ are fixed, the Piola transform formula \eqref{eqn:calMDef} is used to interpret these quantities appropriately on each element $K \ni a$.  These bookkeeping steps are naturally implemented as a precomputation in which the necessary information is encoded into a DOF handler structure.  The precomputation step costs $O(\# \calV_h)$ and does not add significantly to the overall computational cost.  Once this step is completed the rest of the FEM is implemented in a standard way, but using the DOF handler to correctly compute basis functions on each element.  

We now describe more precisely some elements of the precomputation step.
\begin{figure}[h]
\begin{center}
\includegraphics[scale=.25]{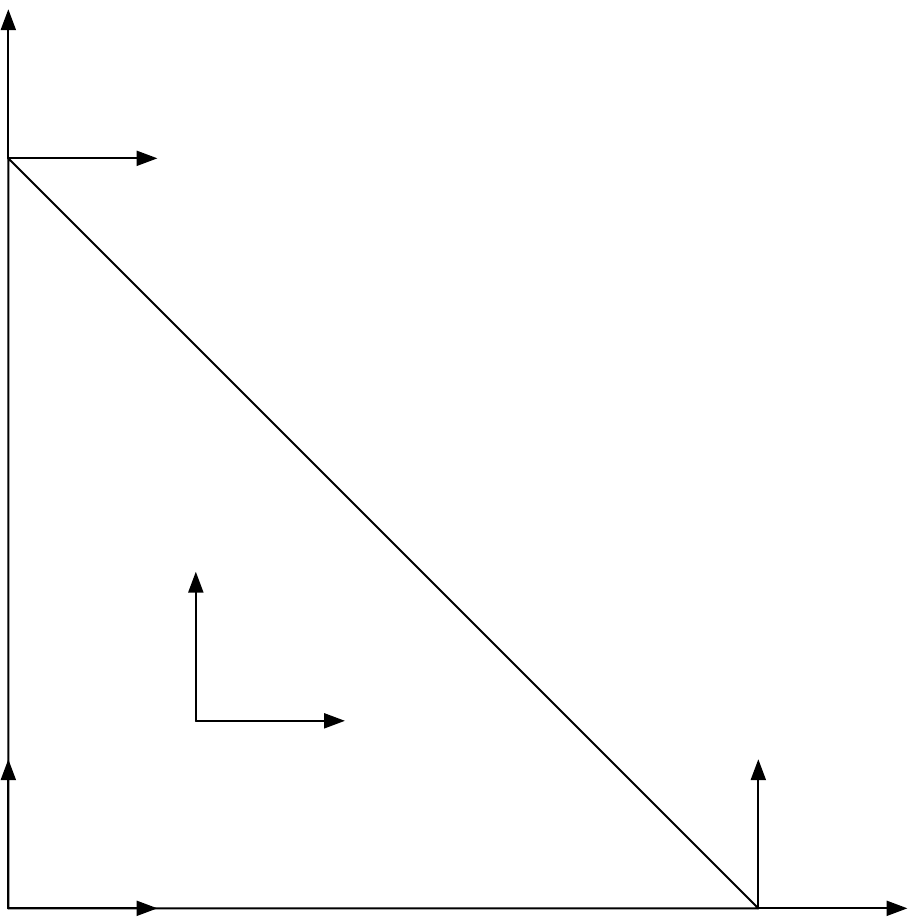}
\put(-23,-8){\parbox[b]{2cm}{$\hat{z}_{2}$}}
\put(-112,-8){\parbox[b]{2cm}{$\hat{z}_{1}$}}
\put(-120,90){\parbox[b]{2cm}{$\hat{z}_{3}$}}
\end{center}
\caption{Degrees of freedom for the reference MINI element}
\label{fig_n1}
\end{figure}
Consider the reference element $\hat{K}$  with associated natural degrees of freedom for the MINI element  (cf.~Figure \ref{fig_n1}). Given a vertex ${\hat z_j} \in \hat{K}$, let {$\hat{\bm \phi}_{1,j}$} and $\hat {\bm \phi}_{2,j}$ be the basis functions corresponding to the vertex degrees 
of freedom in Figure \ref{fig_n1}, i.e., $\hat{\bm \phi}_{1,j}(\hat{z}_i)= \begin{pmatrix} \delta_{ij} \\ 0 \end{pmatrix}$ and $\hat{\bm \phi}_{2,j} (\hat{z}_i)= \begin{pmatrix} 0 \\ \delta_{ij} \end{pmatrix}$.   We translate vertex degrees of freedom from the reference element to physical elements as follows.   For each vertex $a \in \calV_h$:
\begin{enumerate}
 \item Specify a master element $K_a \ni a$.
\item Choose arbitrary unit orthogonal vectors $\bv_{1,a,K_a}$, $\bv_{2,a,K_a}$ lying in the plane containing $K_a$.  
\item For each triangle {$K\in \calT_a$}, compute $\bv_{i,a,K}=\calM_a^K \bv_{i,a,K_a}$, $i=1,2$.
\item For each $K \in \calT_a$, let $j_K \in \{1,2,3\}$ be the local numbering of $a$ in $K$.  
\item For each {$K\in \calT_a$}, let $\calP_{F_K}$ be as in \eqref{ref_piola}.  Solve for {$\alpha_{i,\ell,a,K}$}, $i,\ell\in \{1,2\}$, such that ${\alpha_{i,1,a,K} } \calP_{F_K} \hat {\bm \phi}_{1,j_K}(\hat z_{j_K}) + {\alpha_{i,2,a,K}} \calP_{F_K} \hat {\bm \phi}_{2,j_K}(\hat z_{j_K})=\bv_{i,a,K}$. With degrees of freedom as pictured in Figure \ref{fig_n1}, this expression reduces to the linear system 
\[
\calP_{F_K} \balpha_{i,a,K} = \bv_{i,a,K}\quad \text{with}\quad \balpha_{i,a,K} = 
\begin{pmatrix}
\alpha_{i,1,a,K}\\
\alpha_{i,2,a,K}
\end{pmatrix}.
\]
This system is solved by application of the Moore-Penrose psuedoinverse $(\calP_{F_K}^\intercal \calP_{F_K})^{-1} \calP_{F_K}^{\intercal}.$
\end{enumerate}

The coefficients $\alpha_{i,\ell, a,K}$ serve as a ``Rosetta stone'' (or DOF handler) to translate the individual 
reference basis functions $\hat {\bm \phi}_{i,j}$ elementwise to global basis functions.  
On the element $K$ the global basis functions corresponding to the vertex $a \in K$ are concretely given  by $\varphi_a(\bx) \mathcal{P}_{F_K} \balpha_{1,a,K}$ and $\varphi_a (\bx)\mathcal{P}_{F_K} \balpha_{2,a,K}$, with $\varphi_a$ the standard scalar affine hat function corresponding to $a$.  Recall that in the Euclidean case natural global basis functions corresponding to $a$ are $\varphi_a(\bx) (1,0)^{\intercal}$ and $\varphi_a(\bx) (0,1)^{\intercal}$.  Thus in both the Euclidean and surface cases, the global MINI basis functions may be expressed as the product of a scalar hat function and a vector specifying direction.  However, in the surface case the vectors $\mathcal{P}_{F_K} \balpha_{i,a,K}$ in question are piecewise constant rather than globally constant in order to reflect variation of the tangent plane from element to element.  Once these expressions for global basis functions are in hand, the other aspects of the finite element code are essentially standard.  Note that sparsity patterns for the resulting system matrices are also similar to the Euclidean case, and the system solve generally has similar expense.

\subsection{Numerical results}

We take $\gamma$ to be the ellipsoid given by $\Psi(x,y,z):=\frac{x^2}{1.1^2} + \frac{y^2}{1.2^2} + \frac{z^2}{1.3^2}=1$.  The test solution is 
{$\bu=\bPi (-z^2, x, y)^\intercal$}; cf.~Figure \ref{fig_n2}.  Note that $\bPi={\bf I}-\bnu\otimes \bnu$ with $\bnu=\frac{\nabla \Psi}{|\nabla \Psi |}$ on $\gamma$, so $\bu$ is componentwise a rational function and not a polynomial.  The pressure is $p=xy^3+z$.  The incompressibility condition ${\rm div}_\gamma \bu=0$ does not hold, so the Stokes system must be solved with nonzero divergence constraint.  We employed a MATLAB code built on top of the iFEM library \cite{Ch09PP}.  
\begin{figure}[h]
\begin{center}
\includegraphics[scale=.4]{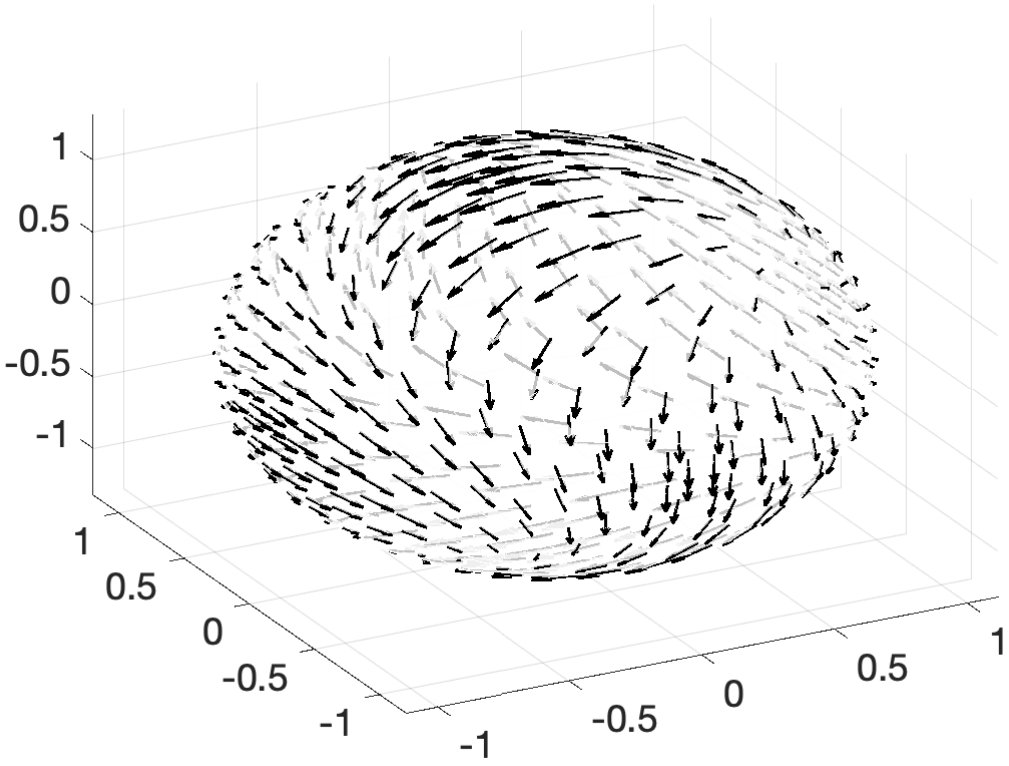}
\end{center}
\caption{Test solution $\bu$}
\label{fig_n2}
\end{figure}

The left plot in Figure \ref{fig_n3} depicts the convergence history for the MINI element on a sequence of uniformly refined meshes.  Optimal convergence is clearly observed in both the energy and $L_2$ norms, in particular $O(h)$ for the energy norm $\|\ipt \bu -\bu_h\|_{H_h^1(\Gamma_h)} + \|p^e -p_h\|_{L_2(\Gamma_h)}$ along with $O(h^2)$ for the error $\|\ipt \bu-\bu_h\|_{L_2(\Gamma_h)}$.  Recall also that the pressure is approximated by affine functions, which can in theory approximate to order $h^2$ in $L_2$. Convergence is generally restricted instead to order $h$ because the pressure is coupled to the velocity $H^1$ norm in the error analysis, but superconvergence of order $h^{3/2}$  may occur on sufficiently structured meshes \cite{ETX11}.  We observe an initial superconvergent decrease of order $h^{3/2}$ or higher, but the expected asymptotic rate of order $h$ is eventually seen; cf.~\cite{PS22} for discussion of similar phenomena in the Euclidean context.  

\begin{figure}[h]
\begin{center}
\includegraphics[scale=.2]{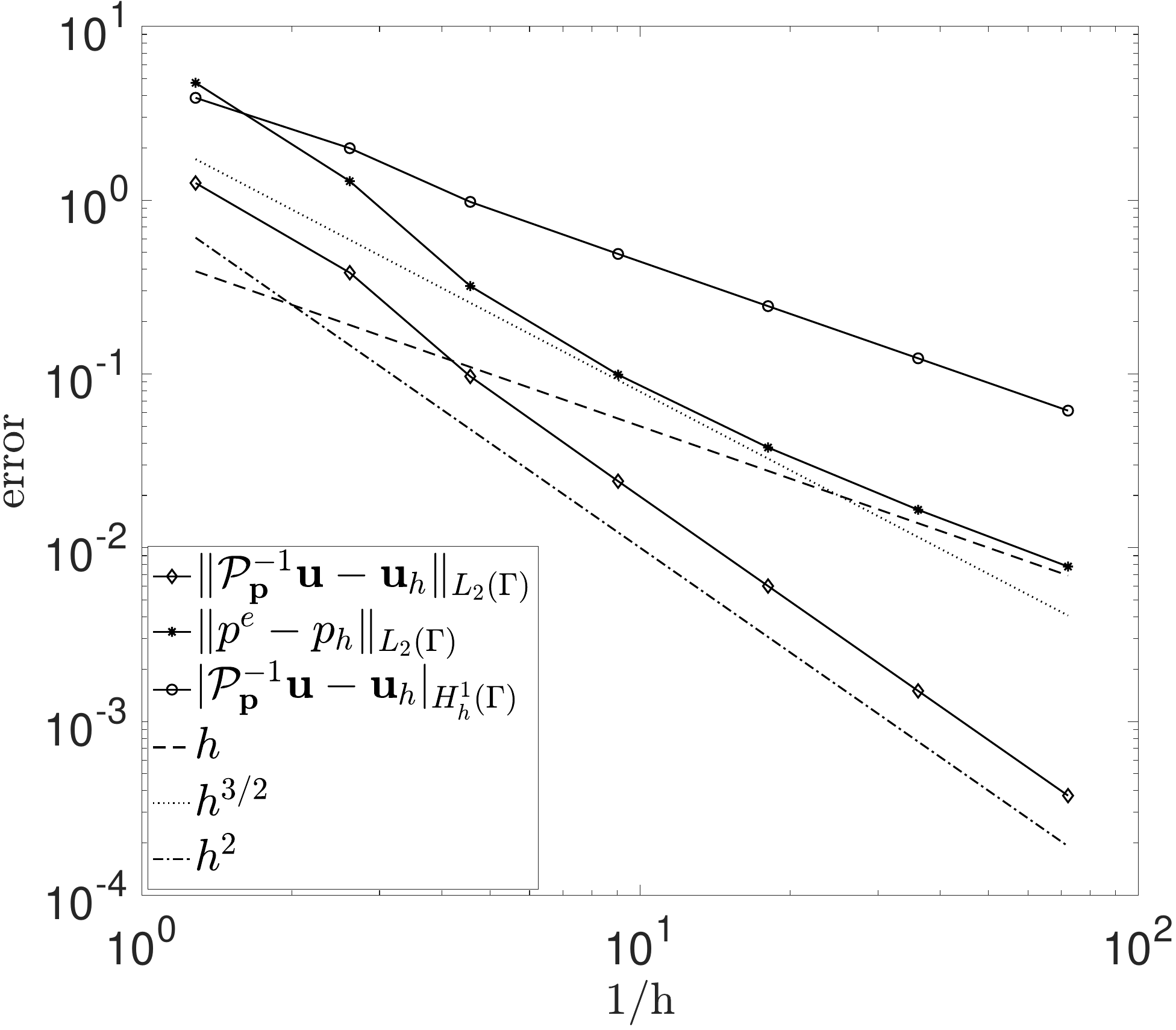}
\includegraphics[scale=.2]{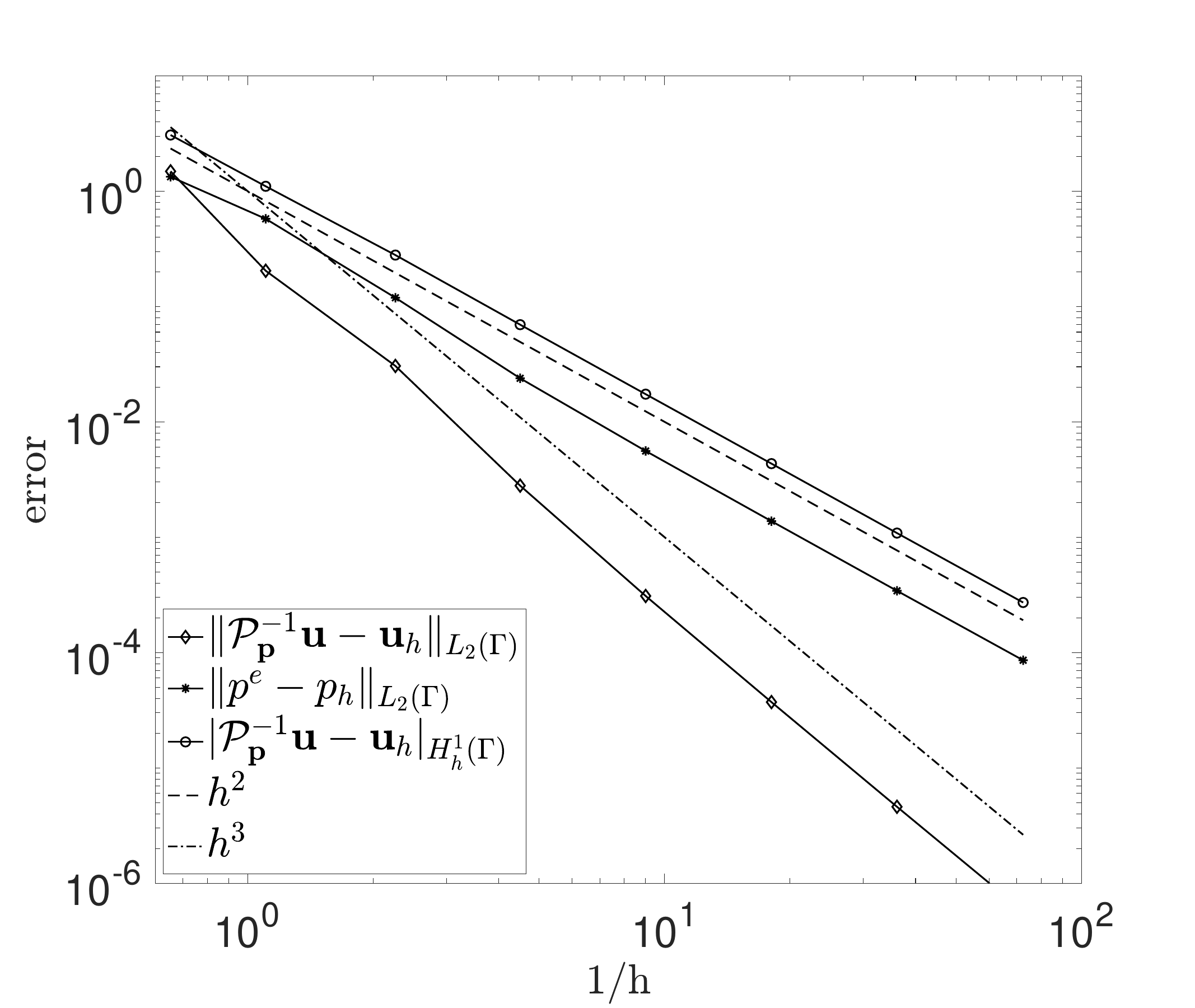}
\end{center}
\caption{Convergence for the MINI element (left) and $\mathbb{P}^2-\mathbb{P}^1$ Taylor-Hood element (right)}
\label{fig_n3}
\end{figure}

We also approximated $(\bu, p)$ using a $\mathbb{P}^2-\mathbb{P}^1$ Taylor-Hood method.  The discrete surface $\Gamma_h$ was taken to be a quadratic rather than affine approximation to $\gamma$ in order to obtain a geometric error commensurate with the expected order of convergence for this element.  Vertex degrees of freedom were defined as above, additionally taking into account the fact that the surface normal on a piecewise quadratic surface is, in contrast to the case of an affine surface, not elementwise constant.  Quadratic Taylor-Hood vector fields have degrees of freedom at edge midpoints in addition to at vertices, and these were defined in a manner completely analogous to the vertex degrees of freedom.  Because
the Piola transform preserves normal continuity, this construction guarantees normal continuity at three points on each (closed) 
edge, thus ensuring $\bH({\rm div}_{\Gamma_h};\Gamma_h)$-conformity (cf.~Proposition \ref{prop:Hdiv}). 
The right plot in Figure \ref{fig_n3} exhibits the expected $O(h^2)$ convergence in the energy norm and $O(h^3)$ convergence for the $L_2$ error in the velocity.  This confirms that our methodology has applicability beyond the MINI element; error analysis and extension to other stable Stokes element pairs employing nodal degrees of freedom will be the subject of future work.  

\appendix \section{Proof of Lemma \ref{lem:DefRelation}}

\begin{proof}
We divide the proof of Lemma \ref{lem:DefRelation} into three steps.

{\em Step 1:}
For a scalar function $q$ defined on $\Gamma_h$, we have the identity \cite[(2.2.19)]{DemlowDziuk07}
\[
\nab_\gamma (q\circ \bp^{-1}) = \big([{\bf I}-d{\bf H}]^{-1} [{\bf I} - \frac{\bnu_h\otimes \bnu}{\bnu_h\cdot \bnu}] \nab_{\Gamma_h} q\big)\circ \bp^{-1}\qquad \text{on }\gamma.
\]
Consequently, for $\bv = (v_1,v_2,v_3)^\intercal  \in \bH_T^1(K)$,
\begin{align*}
(\nab (\bv\circ \bp^{-1})\bPi)_{i,:} 
&= \big(\nab_\gamma (v_i\circ \bp^{-1})\big)^\intercal\\  
&= \Big(\big([{\bf I}-d{\bf H}]^{-1} [{\bf I} - \frac{\bnu_h\otimes \bnu}{\bnu_h\cdot \bnu}] \nab_{\Gamma_h} v_i\big)\circ \bp^{-1}\Big)^\intercal\\
&= \Big((\nab_{\Gamma_h} v_i)^\intercal [{\bf I} - \frac{\bnu_h \otimes \bnu}{\bnu_h\cdot \bnu}]^\intercal 
[{\bf I}-d{\bf H}]^{-\intercal}\Big)\circ \bp^{-1}\\
&= \Big((\nab_{\Gamma_h} v_i)^\intercal [{\bf I} - \frac{\bnu\otimes \bnu_h}{\bnu\cdot \bnu_h}] 
[{\bf I}-d{\bf H}]^{-1}\Big)\circ \bp^{-1}\\
&= \Big((\nab \bv \bPi_h)_{i,:} [{\bf I} - \frac{\bnu\otimes \bnu_h}{\bnu\cdot \bnu_h}] 
[{\bf I}-d{\bf H}]^{-1}\Big)\circ \bp^{-1}.
\end{align*}
Here, with an abuse of notation, we have suppressed the superscript for the extension $\bv^e$.
Thus, we have the identity
\begin{align}\label{eqn:ChainRULE}
\nab (\bv\circ \bp^{-1}) \bPi = \Big(\nab \bv\bPi_h [{\bf I} - \frac{\bnu\otimes \bnu_h}{\bnu\cdot \bnu_h}] 
[{\bf I}-d{\bf H}]^{-1}\Big)\circ \bp^{-1}.
\end{align}
Since $\bv$ is tangential, there holds $\nab \bv= \nab (\bPi_h \bv) = \bPi_h \nab \bv$,
because $\bPi_h$ is  constant on $K$.  Thus,
\begin{align}\label{eqn:nabNabh}
\nab (\bv\circ \bp^{-1}) \bPi = \Big(\nab_{\Gamma_h} \bv [{\bf I} - \frac{\bnu\otimes \bnu_h}{\bnu\cdot \bnu_h}] 
[{\bf I}-d{\bf H}]^{-1}\Big)\circ \bp^{-1}.
\end{align}

{\em Step 2:} Write $\pt \bv = \calP_{\bp} \bv  = ({\bf L} \bv)\circ \bp^{-1}$
with ${\bf L} = {\mu_h^{-1}} [\bPi - d{\bf H}]$.  We then have by \eqref{eqn:nabNabh},
\begin{align*}
\nab_\gamma \pt{\bv} 
&= \bPi \nab \pt{\bv} \bPi = \bPi \nab ({\bf L} \bv\circ \bp^{-1}) \bPi\\
& = \bPi {\bf L} \nab (\bv \circ \bp^{-1}) \bPi + \bPi \nab {\bf L} \bv \circ \bp^{-1} \bPi\\
& = \bPi {\bf L}  \left(\nab_{\Gamma_h} \bv \left[{\bf I} - \frac{\bnu\otimes \bnu_h}{\bnu\cdot \bnu_h}\right] 
[{\bf I}-d{\bf H}]^{-1}\right)\circ \bp^{-1}+ \bPi \nab {\bf L} \bv \circ \bp^{-1} \bPi\\
& =  {\bf L}  \left(\nab_{\Gamma_h} \bv \left[{\bf I} - \frac{\bnu\otimes \bnu_h}{\bnu\cdot \bnu_h}\right] 
[{\bf I}-d{\bf H}]^{-1}\right)\circ \bp^{-1}+ \bPi \nab ({\bf L}\circ \bp^{-1}) \bv \circ \bp^{-1} \bPi,
\end{align*}
where
\begin{equation}\label{eqn:GradMatrix}
(\nab {\bf L} \bv)_{i,j} = \sum_{k=1}^3 \frac{\p {\bf L}_{i,k}}{\p x_j} v_k\qquad i,j=1,2,3.
\end{equation}
We conclude, by adding and subtracting terms, that
\begin{align*}
\nab_\gamma \pt{\bv} 
& = (\nab_{\Gamma_h} \bv) \circ \bp^{-1} + [{\bf L}-\bPi_h] (\nab_{\Gamma_h} \bv)\circ \bp^{-1} \left[{\bf I} - \frac{\bnu\otimes \bnu_h}{\bnu\cdot \bnu_h}\right] [{\bf I} - d{\bf H}]^{-1}\\
&\qquad + (\nab_{\Gamma_h} \bv)\circ \bp^{-1} \left(\left[{\bf I} - \frac{\bnu\otimes \bnu_h}{\bnu\cdot \bnu_h}\right] [{\bf I} - d{\bf H}]^{-1}-\bPi_h\right)+ \bPi \nab ({\bf L}\circ \bp^{-1}) \bv \circ \bp^{-1} \bPi.
\end{align*}

Using $|\bnu-\bnu_h|\lesssim h$, 
$|d|\lesssim h^2$, and \eqref{eqn:muBound}, we have
$|{\bf L}-\bPi_h|\lesssim h$ and 
$|[{\bf I}- \frac{\bnu_h\otimes \bnu}{\bnu\cdot \bnu_h}][{\bf I} - d {\bf H}]^{-1} - \bPi_h|\lesssim h$.
Therefore there holds
\begin{align}\label{eqn:DefDefPre}
|{\rm Def}_{\gamma} \pt \bv-({\rm Def}_{\Gamma_h} \bv)\circ \bp^{-1}|\lesssim  h |(\nab_{\Gamma_h} \bv)\circ \bp^{-1}|
+ |\bPi \nab ({\bf L}\circ \bp^{-1}) \bv\circ \bp^{-1} \bPi|.
\end{align}

{\em Step 3:} In the final step of the proof, we bound 
 the last term in \eqref{eqn:DefDefPre}.
 
Let $\bell^{(r)} = {\bf L}_{:,r}$ denote the $r$th column of ${\bf L}$.
Then \eqref{eqn:GradMatrix} and a short calculation yields
\[
\bPi \nab ({\bf L}\circ \bp^{-1}) \bv\circ \bp^{-1} \bPi =  \sum_{r=1}^3 (\bPi \nab (\bell^{(r)}\circ \bp^{-1}) \bPi) v_r\circ \bp^{-1},
\]
and so, by \eqref{eqn:ChainRULE},
\begin{equation}
\label{eqn:LChain}
\begin{split}
\bPi \nab ({\bf L}\circ \bp^{-1}) \bv\circ \bp^{-1} \bPi 
&= \sum_{r=1}^3 \bPi \left(\nab \bell^{(r)} \bPi_h \left[{\bf I}-\frac{\bnu\otimes \bnu_h}{\bnu\cdot \bnu_h}\right]\big[{\bf I}- d {\bf H}\big]^{-1}v_r\right)\circ \bp^{-1}  \\
&=  \bPi \left(\nab {\bf L} \bv\bPi_h \left[{\bf I}-\frac{\bnu\otimes \bnu_h}{\bnu\cdot \bnu_h}\right]\big[{\bf I}- d {\bf H}\big]^{-1}\right)\circ \bp^{-1}  .
\end{split}
\end{equation}

Taking the derivative of ${\bf L}_{i,k} = \mu^{-1}_h [\bPi_{i,k} - d {\bf H}_{i,k}]$ yields
\begin{align*}
\frac{\p {\bf L}_{i,k}}{\p x_j} 
&= \mu_h^{-1} \left(-{\bf L}_{i,k} \frac{\p \mu_h }{\p x_j} + \frac{\p \bPi_{i,k}}{\p x_j}- \frac{\p d}{\p x_j} {{\bf H}_{i,k}}- d \frac{\p {\bf H}_{i,k}}{\p x_j}\right)\\
%
&= -\mu_h^{-1} \left({\bf L}_{i,k} \frac{\p \mu_h}{\p x_j} +\nu_i {\bf H}_{k,j} +\nu_k {\bf H}_{i,j}+\nu_j {\bf H}_{i,k}+d \frac{\p {\bf H}_{i,k}}{\p x_j}\ \right).
\end{align*}
Thus by \eqref{eqn:muBound} and \eqref{eqn:GradMatrix}, there holds
\begin{equation}\label{eqn:nabL}
\begin{split}
\nab {\bf L} \bv 
&= - \mu^{-1}_h\left[ ({\bf L} \bv)\otimes \nab \mu_h + \bnu \otimes ({\bf H} \bv)+({\bf H} \bv)\otimes \bnu + (\bnu\cdot \bv) {\bf H} + d \nab {\bf H}\bv\right]\\
&= -\left[ ({\bf L} \bv)\otimes \nab \mu_h + \bnu \otimes ({\bf H} \bv)+({\bf H} \bv)\otimes \bnu\right]+O(h|\bv|). 
\end{split}
\end{equation}

Write $\mu_h = \bnu\cdot \bnu_h (1-d \kappa_1)(1-d \kappa_2) = \bnu\cdot \bnu_h \det({\bf I} - d{\bf H})$.
Because $\bnu_h$ is constant on $K$ and ${\bf H} \bnu=0$, there holds $\frac{\p (\bnu\cdot \bnu_h)}{\p x_k} = \bnu_h \cdot \frac{\p \bnu}{\p x_k} = ({\bf H} \bnu_h)_k = ({\bf H}(\bnu_h-\bnu))_k =  O(h)$.
Also by Jacobi's formula and $|d| \lesssim h^2$,
\begin{align*}
\frac{\p }{\p x_k} \det({\bf I}- d {\bf H}) 
&= \det({\bf I}-d {\bf H}){\rm tr}\left(({\bf I} - d {\bf H})^{-1} \frac{\p }{\p x_k} \big({\bf I} - d {\bf H}\big)\right)
 = -\nu_k {\rm tr}({\bf H}) +O(h^2).
\end{align*}
We then conclude using $|1-\bnu \cdot \bnu_h| \lesssim h^2$ that 
\begin{equation}\label{eqn:nabmu}
\nab \mu_h= - (\bnu\cdot \bnu_h)\bnu {\rm tr}({\bf H}) + O(h) =  - \bnu {\rm tr}({\bf H}) + O(h).
\end{equation}
Combining \eqref{eqn:nabL}--\eqref{eqn:nabmu}
yields
\begin{align}\label{eqn:AlmostDone}
\nab {\bf L} \bv 
&= \left[{\rm tr}({\bf H}) ({\bf L} \bv)\otimes  \bnu -\bnu \otimes ({\bf H} \bv)-({\bf H} \bv)\otimes \bnu\right]+O(h|\bv|).
\end{align}

We apply \eqref{eqn:AlmostDone} to \eqref{eqn:LChain}
along with the identity
$\bPi \bnu = \bPi^\intercal \bnu = 0$
and $|\bPi-\bPi_h|\lesssim h$ to obtain
\begin{align*}
|\bPi \nab ({\bf L}\circ \bp^{-1}) \bv\circ \bp^{-1} \bPi|\lesssim h|\bv\circ \bp^{-1}|.
\end{align*}
Combining this with \eqref{eqn:DefDefPre} yields the desired estimate
\begin{align*}
|{\rm Def}_{\gamma} \pt \bv-({\rm Def}_{\Gamma_h} \bv)\circ \bp^{-1}|\lesssim  h \big(|(\nab_{\Gamma_h} \bv)\circ \bp^{-1}|+ |\bv\circ \bp^{-1}|\big).
\end{align*}
\end{proof}


\section*{Acknowledgements} The authors thank Orsan Kilicer for assistance with numerical computations.

\bibliographystyle{siam}
\bibliography{stability_arXiv.bib}

\end{document}